\documentclass[12pt, a4paper]{amsart}%
\usepackage{amscd}
\usepackage{amsmath}
\usepackage{latexsym}
\usepackage{graphicx}
\usepackage{amsfonts}
\usepackage{amssymb}%
\providecommand{\U}[1]{\protect\rule{.1in}{.1in}}
\providecommand{\U}[1]{\protect\rule{.1in}{.1in}}
\nonstopmode
\newcounter{fig}

\theoremstyle{plain}
\newtheorem{Theorem}{Theorem}[section]

\newtheorem{Remarks}[Theorem]{Remarks}

\newtheorem{theorem}[Theorem]{Theorem}

\newtheorem{proposition}[Theorem]{Proposition}
\newtheorem{lemma}[Theorem]{Lemma}
\theoremstyle{definition}

\newtheorem{defn}[equation]{Definition}
\newtheorem{remark}[Theorem]{Remark}
\theoremstyle{remark}

\newtheorem{ex}[equation]{Example}

\begin{document}
\title{Twisted bundles and twisted $K$-theory}
\author[ ]{Max~Karoubi}
\address{Universit\'{e} Denis Diderot- Paris 7, UFR de Math\'{e}matiques. Case 7012,
175, rue du Chevaleret. 75205 Paris cedex 13}
\email{max.karoubi@gmail.com}
\thanks{}
\date{23 October 2010}
\maketitle







\pagestyle{myheadings}
\setcounter{section}{-1}%

\section{Introduction}

Many papers have been devoted recently to twisted $K$-theory as originally
defined in \cite{DK} and \cite{Rosenberg}. See for instance the references
\cite{Atiyah and Segal}, \cite{Karoubi} and the very accessible paper
\cite{Schochet}. We offer here a more direct approach based on the notion of
\textquotedblleft twisted vector bundles\textquotedblright. This is not an
entirely new idea, since we find it in \cite{Gorokhovsky}, \cite{Carey},
\cite{Breen Messing}, \cite{Tsygan} and \cite{Caldararu} for instance, under
different names and from various viewpoints. However, a careful look at this
notion shows that we may interpret such bundles as modules over suitable
algebra bundles. More precisely, the category of twisted vector bundles is
equivalent to the category of vector bundles which are modules over algebra
bundles with fibre $\mathrm{End}(V),$ where $V$ is a finite dimensional vector
space. This notion was first explored in \cite{DK} in order to define twisted
$K$-theory. In the same vein, twisted Hilbert bundles may be used to define
extended twisted $K$-groups, following \cite{Dixmier Douady} and
\cite{Rosenberg}.

More generally, we also analyse the notion of \textquotedblleft twisted
principal bundles\textquotedblright\ with structural group $G$. Under
favourable circumstances, we show that the associated category is equivalent
to the category of locally trivial fibrations, with an action of a bundle of
groups with fibre $G,$ which is simply transitive on each fibre. Such bundles
are classically called \textquotedblleft torsors\textquotedblright\ in the
literature. When the bundle of groups is trivial, we recover the usual notion
of principal $G$-bundle.\ 

As is well known, twisted $K$-theory is a graded group, indexed essentially by
the third cohomology\footnote{More precisely, it is indexed by 3-cocycles. Two
cohomologous cocycles give twisted $K$-groups which are isomorphic (non
canonically). This technical point is discussed in Appendix 8.3.} of the base
space $X$, namely $H^{3}(X;\mathbb{Z}).$ The twisted vector bundles we define
in this paper are also indexed by elements of the same group up to
isomorphism. Roughly speaking, twisted $K$-theory appears as the Grothendieck
group of the category of twisted vector bundles. This provides a geometric
description of this theory, very close in spirit to Steenrod's definition of
coordinate bundles\ \cite{Steenrod}. The more subtle notion of \underline
{graded} twisted $K$-theory, indexed by $H^{1}(X;\mathbb{Z}/2)\times
H^{3}(X;\mathbb{Z})$, may also be analyzed in this framework.

The usual operations on vector bundles (exterior powers, Adams operations...)
are easily extended to twisted vector bundles, in a way parallel to the
operations defined in \cite{Atiyah and Segal}. We have also added a section on
cup-products, in order to show that the various ways to define them coincide
up to isomorphism. This is essentially relevant in the last section of the
paper, where we define an analog of the Chern character.

In this section, we define connections on twisted vector bundles in the finite
and infinite dimensional cases, very much in the spirit of \cite[pg.
78]{Kobayashi-Nomizy}, \cite{Carey}, \cite[Chapitre 1]{Karoubi CRAS}, in a
quite elementary way. It is also described in \cite{Gorokhovsky} and
\cite{Carey-Wang2} with a different method. From this analog of Chern-Weil
theory, we deduce a \textquotedblleft Chern character\textquotedblright\ from
twisted $K$-theory to twisted cohomology. This character is defined in a much
more elaborate way in \cite{Atiyah and Segal 2}, \cite{Tsygan}, \cite{Mathai
Stevenson} and \cite{Tu et Xu} in the general framework of the
\textquotedblleft Connes-Karoubi Chern character\textquotedblright%
\ \cite{Connes}, \cite{Karoubi CRAS}, except in \cite{Atiyah and Segal 2}. In
the paper of Atiyah and Segal \cite{Atiyah and Segal 2}, classical topology
tools are used to show that the twisted Chern character is essentially unique.
Therefore, it coincides with the character defined by our elementary approach
in this paper.

Finally, in a detailed appendix divided into three subsections, we study
carefully the relation between \v{C}ech cohomology with coefficients in
$S^{1}$ and de Rham cohomology. We also discuss more deeply multiplicative
structures and the functorial aspects of twisted $K$-theory and of the Chern character.

\textbf{Aknowledgments}. We thank very much A. Carey, A. Gorokhovsky, J.
Rosenberg, and especially L Breen, C. Schochet and Bai-Ling Wang for their
very relevant comments on preliminary versions of this paper.

\medskip

\textbf{CONTENTS}

1) Twisted principal bundles

2) Relation with torsors

3) Twisted vector bundles

4) Various definitions of twisted $K$-theory

5) Multiplicative structures

6) Thom isomorphism and operations in twisted $K$-theory

7) Connections and the Chern homomorphism

8) Appendix

8.1. Relation between \v{C}ech cohomology with coefficients in $S^{1}$ and de
Rham cohomology

8.2. Some key isomorphisms between various definitions of twisted $K$-groups

8.3. Functoriality of twisted $K$-theory and of the Chern character.

\section{Twisted principal bundles}

Let $G$ be a topological group and let $\mathcal{U=}(U_{i}),i\in I,$ be an
open covering of a topological space $X.$ The \v{C}ech cohomology set
$H^{1}(\mathcal{U};G)$ is well known (see \cite{Steenrod}, \cite{Hirzebruch}
for instance). One starts with \textquotedblleft non abelian\textquotedblright%
\ 1-cocycles $g$, i.e. a set of continuous maps (also called \textquotedblleft
transition functions\textquotedblright)%
\[
g_{ji}:U_{i}\cap U_{j}\longrightarrow G,
\]
such that $g_{kj}\cdot g_{ji}=g_{ki}$ over $U_{i}\cap U_{j}\cap U_{k}.$ Two
cocycles $g$ and $h$ are equivalent if there are continuous maps%
\[
u_{i}:U_{i}\longrightarrow G,
\]
such that%
\[
u_{j}\cdot g_{ji}=h_{ji}\cdot u_{i}.\qquad( 1)\label{morphisms}%
\]
The set of equivalence classes is denoted by $H^{1}(\mathcal{U};G).$ A
covering $\mathcal{V=}(V_{s}),s\in S,$ is a refinement of $\mathcal{U}$ if
there is a map $\tau:S\longrightarrow I$ such that $V_{s}\subset U_{\tau(s)}.
$ We then have a \textquotedblleft restriction map\textquotedblright%
\[
R_{\tau}:H^{1}(\mathcal{U};G)\longrightarrow H^{1}(\mathcal{V};G),
\]
assigning to the $g^{\prime}s$ the functions $k=\tau^{\ast}(g)$ defined by%
\[
k_{s,r}=g_{\tau(s),\tau(r)}.
\]
It is shown in \cite[pg. 48]{Hirzebruch} for instance that the map $R_{\tau}$
is in fact independent of the choice of $\tau.$ We then define%
\[
H^{1}(X;G)=\,\underset{\mathcal{U}}{\text{\textrm{Colim}}}\text{\textrm{\ }%
}H^{1}(\mathcal{U};G),
\]
where $\mathcal{U}$ runs over the \textquotedblleft set\textquotedblright\ of
coverings of $X.$

\medskip

Now let $Z$ be a subgroup of the centre of $G$ and let $\lambda=(\lambda
_{kji})$ be a completely normalized $2$-cocycle of $\mathcal{U}$ with values
in $Z.$ This means that $\lambda=1$ if two of the three indices $k,j,i$ are
equal and that%
\[
\lambda_{\sigma(k)\sigma(j)\sigma(i)}=(\lambda_{kji})^{\varepsilon(\sigma)},
\]
where $\sigma$ is a permutation of the indices $(k,j,i),$ with signature
$\varepsilon(\sigma)$.

\begin{remark}
One can prove (see \cite{Karoubi symetric} for instance) that a \v{C}ech
cocycle in any dimension is cohomologous to a completely normalized one.
Moreover, if every open subset of $X$ is paracompact, any cohomology class may
be represented by a completely normalized \v{C}ech cocycle.
\end{remark}

A $\lambda$-twisted $1$-cocycle (simply called twisted cocycle if $\lambda$ is
implicit) is then given by transition functions $g=(g_{ji})$ as above, such
that%
\[
g_{ii}=1,g_{ji}=(g_{ij})^{-1}%
\]
and
\[
g_{kj}\cdot g_{ji}=g_{ki}\cdot\lambda_{kji}%
\]
over $U_{i}\cap U_{j}\cap U_{k}.$ If we compute the product $g_{lk}\cdot
g_{kj}\cdot g_{ji}$ in two different ways using associativity, we indeed find
that $\lambda$ should be a $2$-cocycle. On the other hand, one can easily show
that the function $g_{ij}\cdot g_{jk}\cdot g_{ki}$ is invariant under a
circular permutation of the indices and is changed to its inverse if we
permute $i$ and $k.$ Since we have $\lambda_{kjk}=1,$ the cocycle $\lambda$
should be completely normalized.

Two twisted cocycles $g$ and $h$ are equivalent if there are continuous maps
$u_{i}:U_{i}\longrightarrow G,$ such that we have a condition analogous to the
above
\[
u_{j}\cdot g_{ji}=h_{ji}\cdot u_{i}\text{ }\qquad(\ref{morphisms})
\]
We define the twisted (non abelian) cohomology $H_{\lambda}^{1}(\mathcal{U}%
;G)$ as the set of equivalence classes.

\begin{proposition}
Let $\mu$ be a $2$-cocycle cohomologous to $\lambda$, i.e. such that we have
the relation%
\[
\mu_{kji}=\lambda_{kji}\cdot\eta_{ji}\cdot\eta_{ki}^{-1}\cdot\eta_{kj},
\]
for some $\eta=(\eta_{ji})$ with $\eta_{ji}=(\eta_{ij})^{-1}$ and $\eta
_{ii}=1.$ Then the map%
\[
\Theta:H_{\lambda}^{1}(\mathcal{U};G)\longrightarrow H_{\mu}^{1}%
(\mathcal{U};G),
\]
sending $(g)$ to the twisted cocycle $(g^{\prime})$ given by $g_{ji}^{\prime
}=g_{ji}\cdot\eta_{ji}$, is an isomorphism.
\end{proposition}

\begin{proof}
If we compute $g_{kj}^{\prime}\cdot g_{ji}^{\prime}$ we indeed find%
\[
g_{kj}^{\prime}\cdot g_{ji}^{\prime}=g_{ki}^{\prime}\cdot\lambda_{kji}%
\cdot\eta_{kj}\cdot\eta_{ji}\cdot(\eta_{ki})^{-1}=g_{ki}^{\prime}\cdot
\mu_{kji}.
\]
This shows that the map $\Theta$ is well defined. The inverse map is of course
given by the correspondence $(g_{ji}^{\prime})\longmapsto(g_{ji}^{\prime}%
\cdot\eta_{ji}^{-1}).$
\end{proof}

\smallskip

From the previous considerations one may define the following category. The
objects are $\lambda$-twisted bundles on a covering $\mathcal{U}$, the
morphisms between $(g_{ji})$ and $(h_{ji})$ being continuous maps $(u_{i})$,
with the compatibility condition (\ref{morphisms}). In this category the
covering $\mathcal{U}$ is fixed together with the $2$-cocycle $\lambda$.

However, this category is too rigid for our purposes, since we want to
consider covering refinements. The covering $\mathcal{V=}$ $(V_{s}),s\in S,$
is a refinement of $\mathcal{U}$ $=(U_{i}),i\in I$ if there is a map
$\tau:S\longrightarrow I$ such that $V_{s}\subset U_{\tau(s)}.$This map $\tau$
induces a morphism%
\[
\Theta_{\tau}:H_{\lambda}^{1}(\mathcal{U};G)\longrightarrow H_{\mu}%
^{1}(\mathcal{V};G)
\]
which is not necessarily an isomorphism. Starting with a twisted cocycle
$(g_{ji})$, its image by $\Theta_{\tau}$ is the cocycle $(h_{sr})$ given by
the formula%
\[
h_{sr}=g_{\tau(s)\tau(r)}.
\]
The $2$-cocycle associated to $h$ is
\[
\mu_{tsr}=g_{\tau(t)\tau(s)}\cdot g_{\tau(s)\tau(r)}\cdot g_{\tau(r)\tau
(t)}=\lambda_{\tau(t)\tau(s)\tau(r)}.
\]

\begin{proposition}
Let $\tau$ and $\tau^{\prime}$ be two maps from $S$ to $I$ such that
$V_{s}\subset U_{\tau(s)}$ and $V_{s}\subset U_{\tau^{\prime}(s)}$ and let $x$
be an element of the set $H_{\lambda}^{1}(\mathcal{U};G).$ Then $\Theta_{\tau
}(x)$ and $\Theta_{\tau^{\prime}}(x)$ are related through an isomorphism%
\[
H_{\mu}^{1}(\mathcal{V};G)\cong H_{\mu^{\prime}}^{1}(\mathcal{V};G),
\]
made explicit in the proof below. This isomorphism does not depend on $x$ and
depends only on $\tau,\tau^{\prime}$ and the $2$-cocycle $\lambda.$
\end{proposition}

\begin{proof}
Let $h^{\prime}$ be the following transition functions%
\[
h_{sr}^{\prime}=g_{\tau^{\prime}(s)\tau^{\prime}(r)}.
\]
We may write%
\[
h_{sr}^{\prime}=g_{\tau^{\prime}(s)\tau(s)}\cdot h_{sr}\cdot g_{\tau
(r)\tau^{\prime}(r)}\cdot\sigma_{sr},
\]
where
\[
\sigma_{sr}=\lambda_{\tau^{\prime}(s)\tau(r)\tau(s)}\cdot\lambda_{\tau
^{\prime}(s)\tau^{\prime}(r)\tau(r)}.
\]

Since we have $h_{rs}=(h_{sr})^{-1},h_{rr}=1$ and the same properties for
$h^{\prime},$ it follows that $\sigma_{rs}=(\sigma_{sr})^{-1}$ and
$\sigma_{rr}=1.$ Therefore,

1) the twisted 1-cocycles $(h_{sr})$ and $(\overline{h}_{sr}),$ where
\[
\overline{h}_{sr}=g_{\tau^{\prime}(s)\tau(s)}\cdot h_{sr}\cdot g_{\tau
(r)\tau^{\prime}(r)},
\]
are isomorphic in the category of twisted bundles over $\mathcal{V}$ with the
same twist.

2) the twisted bundles defined by the 1-cocycles $(\overline{h}_{sr})$ and
$(h_{sr}^{\prime})$ are also isomorphic through the isomorphism
\[
H_{\mu}^{1}(\mathcal{V};G)\cong H_{\mu^{\prime}}^{1}(\mathcal{V};G)
\]
defined in the previous proposition.

We note that $\mu^{\prime}$ is the following $2$-cocycle with values in $Z$%
\[
\mu_{tsr}^{\prime}=\overline{h}_{ts}\cdot\overline{h}_{sr}.\overline{h}%
_{rt}=h_{ts}\cdot h_{sr}\cdot h_{rt}\cdot\sigma_{ts}\cdot\sigma_{sr}%
\cdot\sigma_{rt}=\mu_{tsr}\cdot\sigma_{ts}\cdot\sigma_{sr}\cdot\sigma_{rt},
\]
which is of course cohomologous to $\mu.$

It remains to show that the isomorphism
\[
H_{\mu}^{1}(\mathcal{V};G)\cong H_{\mu^{\prime}}^{1}(\mathcal{V};G)
\]
depends only of $\tau$ and $\tau^{\prime}$ and not of the specific element
$x.$ The previous identity shows indeed that the $2$-cocycles $\mu$ and
$\mu^{\prime}$ are cohomologous through the completely normalized 1-cochain
$\sigma$ which is a function of $\lambda$ only.
\end{proof}

\begin{remark}
Although we don't need it in the proof, this computation showing that $\mu$
and $\mu^{\prime}$ are cohomologous is based on the existence of a twisted
$1$-cocycle $(g_{ji})$ associated to a $2$-completely normalized cocycle
$\lambda.$ Unfortunately, this is not true in general. However, when $X$ has
the homotopy type of a CW-complex, we may also argue as follows in greater
generality. First me may assume that $X$ is pathwise connected, so that we can
choose a base point on $X.$ Now let $PX$ be the path space of $X$ and let
\[
\pi:PX\rightarrow X
\]
be the canonical map associating to a path starting at the base point its end
point. In order to check that $\sigma_{ts}\cdot\sigma_{sr}\cdot\sigma_{rt}%
=\mu_{tsr}^{\prime}\cdot(\mu_{tsr})^{-1},$ we consider the covering of $PX$
defined by the pull-back $\pi^{\ast}(\mathcal{U)}$ of the covering
$\mathcal{U}$ of $X.$ Since the nerve of $\pi^{\ast}(\mathcal{U)}$ is
contractible, there is a completely normalized 1-cochain $\overline{g}$, with
values in the subgroup $Z$ of $G,$ such that $\lambda_{tsr}$ is the associated
twist, i.e. its coboundary. This enables us to perform the previous
computations on $PX$ (with $\overline{g}$ as our 1-twisted cocycle) and hence
on $X,$ since the pull-back of functions from $X$ to $PX$ by the map $\pi$ is injective.
\end{remark}

\section{Relation with torsors}

\medskip

There is another interpretation of twisted principal bundles in some
favourable circumstances and which is more familiar. For this, we observe that
$G$ acts on itself by inner automorphisms and that the kernel of the map%
\[
G\rightarrow Aut(G)
\]
is the centre of $G.$ We now assume that the induced map $G/Z\rightarrow
Aut(G)$ is a homeomorphism on its image for the quotient topology and that the
map%
\[
G\rightarrow G/Z
\]
is a locally trivial fibration. In the applications we have in mind, $G$ is a
Lie group or a Banach Lie group and it is well known that these conditions are
fulfilled if $Z$ is a closed subgroup of the centre.

On the other hand, we notice that if $P$ is a twisted principal bundle
associated to a covering $\mathcal{U}$ with transition functions $g_{ji},$ we
may define a bundle of groups $\mathrm{AUT}(P)$ as follows. Its transition
functions are defined over $U_{i}\cap U_{j}$ by%
\[
g\longmapsto g_{ji}\cdot g\cdot(g_{ji})^{-1}=g_{ji}\cdot g\cdot g_{ij}.
\]

\begin{proposition}
\label{AUT}Let $\widetilde{G}$ be a bundle of groups with fibre $G$ and with
structural group $G/Z,$ acting by inner automorphisms on $G$. Then, if the
covering $\mathcal{U}\mathbb{\ }=(U_{i})$ is fine enough, there is a twisted
principal bundle $P$ such that $\widetilde{G}$ is isomorphic to the bundle of
groups $\mathrm{AUT}(P)$ defined above.
\end{proposition}

\begin{proof}
The bundle of groups $\widetilde{G}$ is given by transition functions%
\[
\gamma_{ji}:U_{i}\cap U_{j}\rightarrow G/Z,
\]
where we may assume without loss of generality that%
\[
\gamma_{ii}=Id\text{ }\quad\text{and\quad}\gamma_{ij}=(\gamma_{ji})^{-1}.
\]
According to our assumptions, the fibration $G\rightarrow G/Z$ is locally
trivial. Therefore, if the covering $\mathcal{U}$ is fine enough, we can find
continuous functions%
\[
g_{ji}:U_{i}\cap U_{j}\rightarrow G
\]
such that the class of $g_{ji}$ is $\gamma_{ji}$, and moreover $g_{ii}=Id$,
$g_{ij}=(g_{ji})^{-1}.$ From these identities, it follows that the following
continuous function defined on $U_{i}\cap U_{j}\cap U_{k}$%
\[
\lambda_{kji}=g_{kj}\cdot g_{ji}\cdot g_{ik}%
\]
is a completely normalized $2$-cocycle with values in $Z.$ Therefore, it
defines a twisted principal bundle $P$ with transition functions $(g_{ji}).$
Moreover, according to the previous considerations, the bundle of groups
$\widetilde{G}$ is canonically isomorphic to $\mathrm{AUT}(P),$ with
transition functions%
\[
u\longmapsto\overline{g}_{ji}(u)=g_{ji}\cdot u\cdot g_{ij}.
\]

\end{proof}

This proposition enables us to relate the category of twisted principal
bundles to more classical mathematical objects. We notice that if $P$ and $Q$
are twisted principal bundles with transition functions $g_{ji}$ and $h_{ji}$
respectively (with the same twist $\lambda),$ we can define a locally trivial
bundle $\mathrm{ISO}(P,Q)$ with fibre $G,$ the transition functions being
automorphisms of the underlying space $G$ defined by%
\[
u\longmapsto h_{ji}\cdot u\cdot g_{ij}=\theta_{ji}(u).
\]
Since we have $g_{kj}\cdot g_{ji}\cdot g_{ik}=h_{kj}\cdot h_{ji}\cdot
h_{ik}=\lambda_{kji},$ the 1-cocycle condition is satisfied for the bundle
$\mathrm{ISO}(P,Q),$ i.e. we have the relation%
\[
\theta_{kj}\cdot\theta_{ji}=\theta_{ki}.
\]
In particular, if $P=Q,$ we get the previous bundle of groups $\mathrm{AUT}%
(P).$

\bigskip Moreover, there is a bundle map%
\[
\mathrm{ISO}(P,Q)\times\mathrm{AUT}(P)\rightarrow\mathrm{ISO}(P,Q).
\]
It is defined by
\[
(u,v)\longmapsto u\circ v,
\]
or by $(u_{i},v_{i})\mapsto u_{i}\circ v_{i}$ in local coordinates. Therefore,
the bundle $\mathrm{ISO}(P,Q)$ inherits a right fibrewise $\mathrm{AUT}%
(P)$-action which is simply transitive on each fibre. In classical
terminology\footnote{It is not the purpose of this paper to develop the theory
of torsors. Roughly speaking, this notion is a generalization of the
definition of a principal bundle $P$. Instead of having a topological group
$G$ acting on $P$ as usual, we have a bundle of groups $\widetilde{G}$ acting
fiberwise on $P$ in a way which is simply transitive on each fiber. In our
situation, the structural group of $\widetilde{G}$ is $G/Z,$ acting on $G$ by
inner automorphisms.}, the bundle $\mathrm{ISO}(P,Q)$ is a
\textrm{\textquotedblleft}torsor\textquotedblright\ over the bundle of groups
$\mathrm{AUT}(P),$ acting on the right.

\begin{theorem}
Let $\widetilde{G}$ be a bundle of groups with fibre $G$ and structural group
$G/Z$ acting on $G$ by inner automorphisms. We assume the existence of a
covering $\mathcal{U}=(U_{i})$ such that $\widetilde{G}$ may be written as
$\mathrm{AUT}(P),$ where $P$ is a $\lambda$-twisted principal bundle. Then,
any torsor $M$ over $\widetilde{G}$ may be written as\textrm{\ }%
$\mathrm{ISO}(P,Q),$ where $Q$ is a $\lambda$-twisted principal bundle. More
precisely, the correspondence $Q\longmapsto\mathrm{ISO}(P,Q)$ induces an
equivalence between the category of $\lambda$-twisted principal bundles and
the category of $\widetilde{G}$-torsors.
\end{theorem}

\begin{proof}
Let $\gamma_{ji}$ be the transition functions of $M$ with fibre $G$ and let
$g_{ji}$ be the transition functions of $P.$ Then the transition functions of
$AUT(P)$ are given by $\overline{g}_{ji}(u)=g_{ji}\cdot u\cdot g_{ij}.$ Now we
claim that the transition functions of $M$ should be of type%
\[
\gamma_{ji}(u)=h_{ji}\cdot u\cdot g_{ij},
\]
for some continuous functions $h_{ji}.$ In order to prove this, we use the
action of $\widetilde{G}$ on the right by writing%
\[
\gamma_{ji}(u)=\gamma_{ji}(1\cdot u)=\gamma_{ji}(1)\cdot\overline{g}%
_{ji}(u)=\gamma_{ji}(1)\cdot g_{ji}.u\cdot g_{ij}.
\]
We then put $h_{ji}=\gamma_{ji}(1)\cdot g_{ji}.$ The fact that $\gamma
_{kj}\cdot\gamma_{ji}=\gamma_{ki}$ implies the identity%
\[
h_{kj}\cdot h_{ji}\cdot u\cdot g_{ij}\cdot g_{jk}=h_{ki}\cdot u\cdot g_{ik}.
\]
Since $g_{ij}\cdot g_{jk}=g_{ik}\cdot\lambda_{ijk},$ this implies that
$h_{kj}.h_{ji}=h_{ki}\cdot(\lambda_{ijk})^{-1}=h_{ki}\cdot\lambda_{kji}$;
therefore the $(h_{ji})$ are the transition functions of a $\lambda$-twisted
principal bundle.

We have to check the coherence of the action of $\widetilde{G}$ on the right,
i.e. the identity%
\[
\gamma_{ji}(u\cdot v)=\gamma_{ji}(u)\cdot\overline{g}_{ji}(v).
\]
This follows from the simple calculation in local coordinates%
\[
\gamma_{ji}(u\cdot v)=h_{ji}\cdot(u\cdot v)\cdot g_{ij}=(h_{ji}\cdot u\cdot
g_{ij})\cdot(g_{ji}\cdot v\cdot g_{ij})=\gamma_{ji}(u)\cdot\overline{g}%
_{ji}(v).
\]

The previous computations show that we can define a functor backwards from the
category of $\widetilde{G}$-torsors to the category of $\lambda$-twisted
principal bundles. It remains to prove that the map%
\[
Hom(Q,Q^{\prime})\rightarrow Hom(\mathrm{ISO}(P,Q),\mathrm{ISO}(P,Q^{\prime}))
\]
is an isomorphism. For this, we analyse the morphisms%
\[
\mathrm{ISO}(P,Q)\rightarrow\mathrm{ISO}(P,Q^{\prime})
\]
which are compatible with the structure of $\mathrm{AUT}(P)$-torsor. Such a
morphism%
\[
\mathrm{ISO}(P,Q)\rightarrow\mathrm{ISO}(P,Q^{\prime})
\]
is given in local coordinates by the formula%
\[
\Phi:u\longmapsto\beta_{i}\cdot u\cdot\alpha_{i},
\]
where $(\alpha_{i})$ (resp. $(\beta_{i}))$ is associated to $\mathrm{AUT}(P)$
(resp. $\mathrm{ISO}(Q,Q^{\prime})).$ We notice the formula%
\[
h_{ji}^{\prime}\cdot\beta_{i}\cdot u\cdot\alpha_{i}\cdot g_{ij}=\beta_{j}\cdot
h_{ji}\cdot u\cdot g_{ij}\cdot\alpha_{j},
\]
where $h_{ji}^{\prime}$ are the coordinate functions of $Q^{\prime}.$ In the
same way, an element of $AUT(P)$ is given in local coordinates by%
\[
\Upsilon:g\longmapsto g\cdot\alpha_{i}.
\]
Therefore, the equation%
\[
\Phi(u\cdot g)=\Phi(u)\cdot\Upsilon(g)
\]
may be written%
\[
\beta_{i}\cdot(u\cdot g)\cdot\alpha_{i}=(\beta_{i}\cdot u\cdot\alpha_{i}%
)\cdot(g\cdot\alpha_{i}),
\]
which is only possible if $\alpha_{i}=1.$
\end{proof}

\begin{remark}
An analog of this theorem in the framework of vector bundles will be proved in
the next section (Theorem \ref{twisted=A-modules}).
\end{remark}

\section{Twisted vector bundles}

One of the main aims of this paper is the theory of \textquotedblleft
twisted\textquotedblright\ vector bundles\footnote{For simplicity's sake, we
shall only consider complex vector bundles. The theory for real or
quaternionic vector bundles follows the same pattern. More generally, we may
also consider vector bundles with fibres finitely generated projective modules
over a Banach algebra. This remark will be useful in the next section for
$\mathcal{A}$-bundles.}. We essentially studied it in Section 1, with the
structural group $G=$ $GL_{n}(\mathbb{C}).$ However, to keep track of the
linear structure and because we want the \textquotedblleft
fibres\textquotedblright\ not to have the same dimension on each connected
component of $X$, we change slightly the general definition as follows.

We start as before with a covering $\mathcal{U}$ $=(U_{i}),i\in I,$ together
with a finite dimensional vector space $E_{i}$ \textquotedblleft
over\textquotedblright\ $U_{i}.$ Another piece of information is a completely
normalized $2$-cocycle $\lambda_{kji}$ with values in $\mathbb{C}^{\times}.$ A
$\lambda$-twisted vector bundle $E$ on $X$ is then defined by transition
functions%
\[
g_{ji}:U_{i}\cap U_{j}\rightarrow\mathrm{Iso}(E_{i},E_{j}),
\]
such that%

\[
g_{ii}=1,g_{ji}=(g_{ij})^{-1}%
\]
and
\[
g_{kj}\cdot g_{ji}=g_{ki}\cdot\lambda_{kji},
\]
as in the previous section. There is however a slight change for the
definition of morphisms from a twisted vector bundle $E$ to another one $F,$
with the same twist $\lambda.$ They are defined as continuous maps%
\[
u_{i}:U_{i}\rightarrow\mathrm{Hom}(E_{i},F_{i}),
\]
such that
\[
u_{j}\cdot g_{ji}=h_{ji}\cdot u_{i}.
\]
The point is that we no longer require the $u_{i}$ to be isomorphisms.

More generally, let $E$ be a $\lambda$-twisted vector bundle on a covering
$\mathcal{U}$ with transition functions $(g_{ji})$ and let $F$ be a $\mu
$-twisted vector bundle on the same covering with transition functions
$(h_{ji})$. We define a $\lambda^{-1}\cdot\mu$-twisted vector bundle in the
following way: over each $U_{i}$ we take as \textquotedblleft
fibre\textquotedblright\ $\mathrm{Hom}(E_{i},F_{i})$ and as transition
functions the isomorphisms%
\[
\mathrm{Hom}(E_{i},F_{i})\rightarrow\mathrm{Hom}(E_{j},F_{j}),
\]
defined by%
\[
\theta_{ji}:f_{i}\longmapsto h_{ji}\circ f_{i}\circ g_{ij}=f_{j}.
\]
We denote this twisted vector bundle by $\mathrm{HOM}(E,F).$ An interesting
case is when $E$ and $F$ are associated to the same $2$-cocycle $\lambda.$
Then $\mathrm{HOM}(E,F)$ is a genuine vector bundle associated to $Hom(E,F)$
by the following proposition.

\begin{proposition}
Let $E$ and $F$ be two $\lambda$-twisted vector bundles. Then the vector space
of morphisms from $E$ to $F,$ i.e. $Hom(E,F),$ may be identified canonically
with the vector space of sections of the vector bundle $\mathrm{HOM}(E,F).$
\end{proposition}

\begin{proof}
A section of this vector bundle is defined by elements $f_{i}$ of
$Hom(E_{i},F_{i})$ such that%
\[
\theta_{ji}(f_{i})=f_{j}.
\]
This relation is translated as%
\[
h_{ji}\circ f_{i}=f_{j}\circ g_{ji},
\]
which is exactly the definition of morphisms from $E$ to $F.$
\end{proof}

An interesting case of the previous proposition is when $E=F,$ so that
$\mathrm{HOM}(E,E)=\mathrm{END}(E)$ is an algebra bundle A. The following
theorem relates algebra bundles to twisted vector bundles.

\begin{theorem}
\label{A=End(E)}Any algebra bundle $A$ with fibre $\mathrm{End}(V),$ where $V
$ is a finite dimensional vector space of positive dimension, is isomorphic to
some $\mathrm{END}(E),$ where $E$ is a twisted vector bundle on a suitably
fine covering of $X.$
\end{theorem}

\begin{proof}
Let $V=\mathbb{C}^{n}.$ According to the Skolem-Noether Theorem, the
structural group of $A$ is $\mathrm{PGL}_{n}(\mathbb{C})=\mathrm{GL}%
_{n}(\mathbb{C})/\mathbb{C}^{\times},$ where $\mathrm{PGL}_{n}(\mathbb{C})$
acts on $M_{n}(\mathbb{C})$ by inner automorphisms. We may describe this
bundle $A $ by transition functions%
\[
\gamma_{ji}:U_{i}\cap U_{j}\rightarrow\mathrm{PGL}_{n}(\mathbb{C}),
\]
for a suitable covering $\mathcal{U}=(U_{i})$ of $X.$ Without loss of
generality, we may assume that $\gamma_{ii}=1$ and that $\gamma_{ji}%
=(\gamma_{ij})^{-1}.$ On the other hand, the principal fibration%
\[
\mathrm{GL}_{n}(\mathbb{C})\rightarrow\mathrm{PGL}_{n}(\mathbb{C})
\]
admits local continuous sections. Therefore, if we choose the covering
$\mathcal{U}$ $=(U_{i})$ fine enough, we can lift these $\gamma_{ji}$ to
continuous functions%
\[
g_{ji}:U_{i}\cap U_{j}\rightarrow\mathrm{GL}_{n}(\mathbb{C}).
\]
Moreover, we may choose the $g_{ji}$ such that $g_{ii}=1,g_{ij}=(g_{ji})^{-1}
$. Therefore, we have the identity $g_{kj}\cdot g_{ji}=g_{ki}\cdot
\lambda_{kji}$, where%
\[
\lambda_{kji}:U_{i}\cap U_{j}\cap U_{k}\rightarrow\mathbb{C}^{\times}%
\]
is de facto a completely normalized $2$-cocycle. If $E$ is the twisted vector
bundle associated to the $g^{\prime}s,$ we see that the algebra bundle
$\mathrm{END}(E)$ has transition functions which are%
\[
f\longmapsto g_{ji}\circ f\circ(g_{ji})^{-1},
\]
i.e. the inner automorphisms associated to the $g_{ji}.$
\end{proof}

\begin{remark}
\label{fine}We shall assume from now on that the coverings $\mathcal{U}$ we
are considering are \textquotedblleft good\textquotedblright. This means that
$\mathcal{U}$ has a finite number of elements and that all possible
intersections of elements of $\mathcal{U}$ are either empty or contractible.
This is always possible if $X$ is a compact manifold as shown for instance in
\cite{Bott-Tu} and \cite{Karoubi-Leruste}. In the previous theorem, we are
then able to replace the words \textquotedblleft suitably
fine\textquotedblright\ by \textquotedblleft good\textquotedblright\ since the
fibration%
\[
\mathrm{GL}_{n}(\mathbb{C})\rightarrow\mathrm{PGL}_{n}(\mathbb{C})
\]
has the homotopy lifting property. In this case, we also have
\[
H^{\ast}(X)\cong H^{\ast}(N(\mathcal{U})),K(X)\cong K(N(\mathcal{U})),
\]
etc., where $N(\mathcal{U})$ is the nerve of the covering $\mathcal{U}.$ Note
that its geometric realization has the homotopy type of $X.$
\end{remark}

\begin{remark}
For most spaces we are considering, good coverings are cofinal: any open
covering as a good refinement. This is the case for finite polyedra and, more
geometrically, for compact riemannian manifolds with open geodesic coverings
\cite{Bott-Tu}.
\end{remark}

The previous considerations also show that the cohomology class in%
\[
H^{2}(X;\mathbb{C}^{\times})\cong H^{3}(X;\mathbb{Z})
\]
associated to a twisted vector bundle is a torsion class (assuming that the
covering is good as in Remark \ref{fine}). To prove this, we consider the
commutative diagram%
\[%
\begin{array}
[c]{ccccccc}%
1\rightarrow & \mu_{n} & \rightarrow & \mathrm{U}(n) & \rightarrow &
\mathrm{PU}(n) & \rightarrow1\\
& \downarrow &  & \downarrow &  & \downarrow & \\
1\rightarrow & \mathbb{C}^{\times} & \rightarrow & \mathrm{GL}_{n}%
(\mathbb{C}) & \rightarrow & \mathrm{PGL}_{n}(\mathbb{C}) & \rightarrow1
\end{array}
\]
The non abelian cohomologies $H^{1}(X;PU(n))$ and $H^{1}(X;PGL_{n}(C))$ are
isomorphic and the coboundary map%
\[
H^{1}(X;\mathrm{PU}(n))\cong H^{1}(X;\mathrm{PGL}_{n}(\mathbb{C}))\rightarrow
H^{2}(X;\mathbb{C}^{\times})\cong H^{3}(X;\mathbb{Z})
\]
factors through $H^{2}(X;\mu_{n})$ (also see Appendix 8.1). Therefore the
cocycle $(\lambda_{kji})$ defines a torsion class in $H^{3}(X;\mathbb{Z}).$ It
is a theorem of Serre \cite{Grothendieck Bourbaki} that such an element comes
from an algebra bundle as we have described. Later on, we shall show how we
can recover the full cohomology group $H^{3}(X;\mathbb{Z})$ from algebra
bundles of infinite dimension, as it was observed by Rosenberg
\cite{Rosenberg}.

The following theorem is important for our dictionary relating twisted vector
bundles to modules over suitable algebra bundles.

\begin{theorem}
\label{twisted=A-modules}Let A be an algebra bundle which may be written as
$\mathrm{END}(E),$ where $E$ is a twisted vector bundle associated to a
covering $\mathcal{U}$, transition functions $g_{ji}$ and a completely
normalized $2$-cocycle $\lambda$ with values in $\mathbb{C}^{\times}.$ Let
$\mathcal{E}_{\lambda}(\mathcal{U})$ be the category of $\lambda$-twisted
vector bundles and $\mathcal{E}^{A}(\mathcal{U})$ be the category of finite
dimensional vector bundles trivialized by the covering $\mathcal{U}$, which
are right A-modules. Then the functor%
\[
\psi:\mathcal{E}_{\lambda}(\mathcal{U})\rightarrow\mathcal{E}^{A}(\mathcal{U})
\]
defined by%
\[
F\mapsto\mathrm{HOM}(E,F)
\]
is an equivalence of categories.
\end{theorem}

\begin{proof}
We first notice that if $M,N$ and $P$ are finite dimensional vector spaces
with $M\neq0$ and if $\Lambda=End(M)$, the obvious map%
\[
Hom(N,P)\rightarrow Hom_{\Lambda}(Hom(M,N),Hom(M,P))
\]
is an isomorphism. Since $N$ is a direct summand of some $M^{r},$ it is enough
to check the statement for $N=M,$ in which case it is obvious. This functorial
isomorphism at the level of vector spaces may be translated into the framework
of twisted vector bundles by the isomorphism%
\[
Hom(F,G)\overset{\cong}{\longrightarrow}Hom_{A}(\mathrm{HOM}(E,F),\mathrm{HOM}%
(E,G)).
\]
This shows that the functor $\Psi$ is fully faithful.

On the other hand, we have a canonical isomorphism of vector spaces%
\[
Hom(M,N)\otimes_{A}M\rightarrow N,
\]
defined by $(f,x)\longmapsto f(x)$ which can also be translated in the
framework of twisted vector bundles. This shows that if we start with a bundle
$L$ which is a right $A$-module, where $A$ is some $\mathrm{END}(E),$ we can
associate to it a twisted vector bundle $F$ by the formula%
\[
F=L\otimes_{A}E=\Psi^{\prime}(L).
\]
Since\textrm{\ }$\mathrm{HOM}(E,F)\otimes_{A}E$ is canonically isomorphic to
$F,$ $\psi^{\prime}$ induces a functor going backwards%
\[
\psi^{\prime}:\mathcal{E}^{A}(\mathcal{U})\rightarrow\mathcal{E}_{\lambda
}(\mathcal{U}).
\]
Finally, there is an obvious isomorphism%
\[
L\rightarrow\mathrm{HOM}(E,L\otimes_{A}E)=\Psi(\Psi^{\prime}(L))
\]
This shows that the functor $\Psi$ is essentially surjective.
\end{proof}

This module interpretation enables us to prove the following Theorem.

\begin{theorem}
Let $\mathcal{U}$ $=(U_{i}),i\in I,$ be a good\ covering of $X$ as in Remark
\ref{fine} and let $\mathcal{V}$ $=(V_{s}),$ $s\in S,$ be a refinement of
$\mathcal{U}$ which is also good. Then for any $\tau:S\rightarrow I$ such that
$V_{s}\subset U_{\tau(s)},$ the associated restriction map%
\[
R_{\tau}:H_{\lambda}^{1}(\mathcal{U};GL_{n}(\mathbb{C}))\rightarrow
H_{\tau^{\ast}(\lambda)}^{1}(\mathcal{V};GL_{n}(\mathbb{C}))
\]
is a bijection.
\end{theorem}

\begin{proof}
Since $\mathcal{U}$ is good, for any completely normalized cocycle $\lambda,$
one can find a twisted vector bundle $E$ of rank $m$ on $\mathcal{U}$ and
$\mathcal{V}$, such that $A=\mathrm{END}(E)$ is a bundle of algebras
associated to $\lambda.$ According to the previous equivalence of categories,
the sets $H_{\lambda}^{1}(\mathcal{U};GL_{n}(\mathbb{C})$) and $H_{\tau^{\ast
}(\lambda)}^{1}(\mathcal{V};GL_{n}(\mathbb{C}))$ are in bijective
correspondence with the set of $A$-modules which are locally of type
$Hom(\mathbb{C}^{m},\mathbb{C}^{n}).$ With this identification, the
restriction map $R$ is just an automorphism of this set.
\end{proof}

\begin{remark}
We may prove the homotopy invariance of the category of twisted vector bundles
thanks to this dictionary (at least if $X$ is compact): a twisted vector
bundle may be interpreted as a bundle of $A$-modules, or as a finitely
generated projective module over the Banach algebra $\Lambda=\Gamma(X,A)$ of
continuous sections of $A.$ It is easy to show that modules over
$\Lambda\left[  0,1\right]  $ can be extended from $\Lambda$ (see e.g.
\cite{Karoubi livre}).
\end{remark}

\bigskip

\section{Twisted $K$-theory}

\bigskip

Let $\mathcal{U}$ be a good covering (Remark \ref{fine}) of a space $X$ and
let $\lambda_{kji}$ be a completely normalized $2$-cocycle with values in
$\mathbb{C}^{\times}.$ We consider the category of twisted vector bundles
associated to $\mathcal{U}$ and to the cocycle $\lambda$. This is clearly an
additive category which is moreover pseudo-abelian (every projection operator
has a kernel). We denote by $K_{\lambda}(\mathcal{U})$ its Grothendieck group,
which is also the $K$-group of the category of $A$-modules over $X,$ where
$A\mathcal{=}$ $\mathrm{END}(E)$, as explained at the end of the previous
section. Since this definition is independent from $\mathcal{U}$ up to a non
canonical isomorphism (see Appendix 8.3), we shall also call it $K_{\lambda
}(X)$: this is the classical definition of (ungraded) twisted $K$-theory as
detailed in many references, e.g. \cite{DK}, \cite{Atiyah and Segal},
\cite{Karoubi}.

In this situation, the cocycle $\lambda$ has a cohomology class $\left[
\lambda\right]  $ in the torsion subgroup of
\[
H^{2}(X;\mathbb{C}^{\times})\cong H^{3}(X;\mathbb{Z}).
\]
as we saw in Section 2. When $\left[  \lambda\right]  $ is not necessarily a
torsion class, we should consider \textquotedblleft twisted Hilbert
bundles\textquotedblright\ which are defined in the same way as twisted vector
bundles but with a fibre which is an infinite dimensional Hilbert
space\footnote{For simplicity's sake, we assume $H$ to be separable, i.e.
isomorphic to the classical $l^{2}$ space.} $H.$ It is also more convenient to
use the unitary group $U(H)$ instead of the general linear group as our basic
structural group. In other words, the $(g_{ji})$ in Sections $1$ and $2$ are
now elements of $U(H).$ The 2-cocycle $(\lambda_{kji})$ takes its values in
the topological group $S^{1}.$

From the fibration%
\[
S^{1}\rightarrow U(H)\rightarrow PU(H)
\]
and the contractibility of $U(H)$ (Kuiper's theorem), we see that $PU(H)$ is a
model of the Eilenberg-Mac Lane space $K(\mathbb{Z},3).$ On the other hand,
since $PU(H)$ acts on $\mathcal{L}(H)=\mathrm{End}(H)$ by inner automorphisms,
we deduce that any $2$-cocycle $\lambda=(\lambda_{kji})$ defines an algebra
bundle \underline{$\mathcal{L}$}$_{\lambda}$ with fibre $\mathcal{L}(H)$ which
is well defined up to isomorphism. Therefore, as in the finite dimensional
case, we have the following theorem.

\begin{theorem}
Let \underline{$\mathcal{L}$}$_{\lambda}$ be the bundle of algebras with fibre
$\mathcal{L}(H)$ associated to the cocycle $\lambda.$ Then, if the covering
$\mathcal{U}$ is good as in Remark \ref{fine}, \underline{$\mathcal{L}$%
}$_{\lambda}$ may be written as $\mathrm{END}(E),$ where $E$ is a $\lambda
$-twisted Hilbert bundle$.$
\end{theorem}

\begin{proof}
We just copy the proof of Theorem \ref{A=End(E)} in the infinite dimensional
case. In a more precise way, the structural group of \underline{$\mathcal{L}$%
}$_{\lambda}$ is $PU(H)=U(H)/S^{1}$ acting on $\mathcal{L}(H)$ by inner
automorphisms. Therefore, we may describe the principal bundle by transition
functions%
\[
\gamma_{ji}:U_{i}\cap U_{j}\rightarrow PU(H)
\]
for a good covering $\mathcal{U}=(U_{i})$ of $X$ and we may assume that
$\gamma_{ii}=1,\gamma_{ji}=(\gamma_{ij})^{-1}.$ Since the principal fibration%
\[
\pi:U(H)\rightarrow PU(H)
\]
is locally trivial and the $U_{i}\cap U_{j}$ are contractible (if non empty),
there are continuous maps%
\[
g_{ji}:U_{i}\cap U_{j}\rightarrow U(H),
\]
such that $\pi\circ g_{ji}=\gamma_{ji}.$ The proof now ends as the proof of
Theorem \ref{A=End(E)}.
\end{proof}

\begin{theorem}
Let \underline{$\mathcal{L}$}$_{\lambda}$ be the algebra bundle $\mathrm{END}%
(E),$ where $E$ is a $\lambda$-twisted Hilbert bundle on a covering
$\mathcal{U}.$ Let $\mathcal{E}_{\lambda}(\mathcal{U})$ be the category of
$\lambda$-twisted Hilbert bundles with fibre $H$ and, finally, let
$\mathcal{E}^{\underline{\mathcal{L}}_{\lambda}}(\mathcal{U})$ be the category
of bundles which are right \underline{$\mathcal{L}$}$_{\lambda}$%
-modules\footnote{More precisely, we assume that locally the module is
isomorphic to $\mathcal{L}(H),$ with its standard $\mathcal{L}(H)$-module
structure.}, trivialized over the elements of $\mathcal{U}$. Then, the functor%
\[
\Psi:\mathcal{E}_{\lambda}(\mathcal{U})\rightarrow\mathcal{E}^{\underline
{\mathcal{L}}_{\lambda}}(\mathcal{U}),
\]
defined by the formula%
\[
F\longmapsto\mathrm{HOM}(E,F).
\]
is an equivalence of categories.
\end{theorem}

\begin{proof}
It is also completely analogous to the proof of Theorem
\ref{twisted=A-modules}. In a more precise way, instead of considering all
finite dimensional vector spaces, we take Hilbert spaces $M,N,P,etc.$ of the
same cardinality, i.e. isomorphic to the classical $l^{2}$-space. For
instance, the isomorphism used in the proof of Theorem \ref{twisted=A-modules}%
\[
Hom(F,G)\overset{\cong}{\longrightarrow}Hom_{A}(\mathrm{HOM}(E,F),\mathrm{HOM}%
(E,G))
\]
is a consequence of the fact that it is true at the level of Hilbert spaces
since $Hom(M,N)$ is isomorphic to $End(M)=\mathcal{L}(H).$ The proof of the
theorem again ends as in the case of finite dimensional vector spaces.
\end{proof}

For $\left[  \lambda\right]  \in H^{3}(X;\mathbb{Z}\mathrm{)}=H^{1}(X;PU(H))$
which is not necessarily a torsion class, we may define the associated twisted
$K$-theory in many ways. The first definition is due to Rosenberg
\cite{Rosenberg}: the class $\left[  \lambda\right]  $ is represented up to
isomorphism by a principal bundle $P$ with structural group $PU(H).$ Since
$PU(H)$ is acting on the ideal of compact operators $\mathcal{K}$ in
$\mathcal{L=L}(H)$ by inner automorphisms, we get an associated bundle
\underline{$\mathcal{K}$}$_{\lambda}$ of C*-algebras. The twisted $K$-theory
is then the usual $K$-theory of the algebra of sections of \underline
{$\mathcal{K}$}$_{\lambda}.$An equivalent way to define \underline
{$\mathcal{K}$}$_{\lambda}$ is to consider a twisted Hilbert bundle $E$
associated to the cocycle $\lambda$ (it is unique up to isomorphism). Then,
\underline{$\mathcal{K}$}$_{\lambda}$ is the subalgebra of the sections of the
bundle \underline{$\mathcal{L}$}$_{\lambda}=END(E)$ which belong to
$\mathcal{K}(H)$ over each open set of $\mathcal{U}.$

One unpleasant aspect of this definition is the non existence of a unit
element in \underline{$\mathcal{K}$}$_{\lambda},$ which makes its $K$-theory
slightly complicated to handle. However, we may replace $\mathcal{K}$ by the
subalgebra $\mathcal{A}$ of $\mathcal{L\times L}$ consisting of couples of
operators $(f,g)$ such that $f-g$ $\in\mathcal{K}.$ The group $PU(H)$ is
acting on $\mathcal{A}$, so that we may also twist the algebra $\mathcal{A}$
by $\lambda$ in order to get an algebra bundle \underline{$\mathcal{A}$%
}$_{\lambda}.$ The obvious exact sequence of C*-algebras%
\[
0\rightarrow\mathcal{K}\rightarrow\mathcal{A}\rightarrow\mathcal{L}%
\rightarrow0
\]
induces an exact sequence of algebra bundles%
\[
0\rightarrow\underline{\mathcal{K}}_{\lambda}\rightarrow\underline
{\mathcal{A}}_{\lambda}\rightarrow\underline{\mathcal{B}}_{\lambda}%
\rightarrow0.
\]
Here and elsewhere, using a variation of the Serre-Swan theorem, we shall
often use the same terminology for an algebra bundle and its associated
algebra of continuous sections. In particular the $K$-theory of \underline
{$\mathcal{A}$}$_{\lambda}$ is canonically isomorphic to the $K$-theory of
$\underline{\mathcal{K}}_{\lambda}$ since $\underline{\mathcal{B}}_{\lambda}$
is a flabby algebra\footnote{A Banach algebra $A$ is called flabby if there is
a topological $A$-bimodule $M$ which is projective of finite type as a right
module, such that $M\oplus A$ is isomorphic to $M.$ This is equivalent to
saying that the Banach category $\mathcal{C=P}(A)$ is flabby: there is a
linear continuous functor $\tau$ from $\mathcal{C}$ to itself such that
$\tau\oplus Id_{\mathcal{C}}$ is isomorphic to $\tau.$} (in particular its
$K$-groups are trivial).

A comment is in order to make our previous definition more functorial: the
$\lambda$-twisted $K$-theory is defined precisely as the $K$-theory of bundles
with fibres $\mathcal{A}$-modules which are finitely generated and projective
but twisted by the cocycle $\lambda.$ How this depends only on the cohomology
class $\left[  \lambda\right]  $ is discussed in Appendix 8.3. Our Section $3$
on twisted vector bundles may now be rewritten by replacing the field of
complex numbers $\mathbb{C}$ by the C*-algebra $\mathcal{A}$ and the finite
dimensional bundles by \textquotedblleft$\mathcal{A}$%
-bundles\textquotedblright\ as above. Theorem $3.5$ adapted to this situation
shows that the category of $\lambda$-twisted $\mathcal{A}$-bundles is
equivalent to the category of \underline{$\mathcal{A}$}$_{\lambda}$-modules if
the covering $\mathcal{U}$ of $X$ is good. This shows in particular that the
theory of twisted $\mathcal{A}$-bundles is homotopically invariant (at least
if $X$ is compact).

However, one has to point out a main difference between $\mathbb{C}$-modules
and $\mathcal{A}$-modules: a priori, the fibres of $\mathcal{A}$-bundles are
not necessarily free\footnote{However, we shall show in Section 7 that the
fibres are free modules if the restriction of the cohomology class of
$\lambda$ to every connected component of $X$ is of infinite order.}. However,
since $K(\mathcal{A})$ is canonically isomorphic to $\mathbb{Z}$, each
$\mathcal{A}$-bundle $E$ induces a locally constant function (called the
\textquotedblleft rank\textquotedblright)%
\[
Rk:X\rightarrow\mathbb{Z}\mathrm{,}%
\]
obtained by applying the $K$-functor to each fibre. This correspondence
defines a group map%
\[
Ch_{(0)}:K(\underline{\mathcal{A}}_{\lambda})\rightarrow H^{0}(X;\mathbb{Z}).
\]
In Section 7 we shall see how to define \textquotedblleft higher Chern
characters\textquotedblright\ $Ch_{(m)},$ starting from this elementary step.

In the spirit of Section 1, we may also consider twisted principal
$G$-bundles, where $G$ is the group of invertible elements in the algebra
$\mathcal{A}$. We note that the elements of $G$ are couples of invertible
operators $(g,h)$ in a Hilbert space such that $g-h$ is compact. We get
elements in the centre by considering $g=h\in\mathbb{C}^{\times}.$ More
accurately, one should replace $\mathcal{A}$ by the sub-algebra $\mathrm{End}%
(P),$ where $P$ is a finitely generated projective $\mathcal{A}$-module which
is the fibre of the bundles we are considering (assuming the base is
connected; otherwise the fibre $P$ may vary). Then $G$ is not exactly
$\mathcal{A}^{\times}$ but the subgroup $\mathrm{Aut}(P)$ of $\mathcal{A}%
^{\ast}.$ This point of view will be exploited in Section 7 for the definition
of the Chern character, whose target is twisted cohomology.

Finally, there is a third definition of twisted $K$-theory in terms of
Fredholm operators, following the ideas in \cite{Atiyah Book}, \cite{Janich}
and \cite{DK}. We consider the set of homotopy classes of triples%
\[
(E_{0},E_{1},D),
\]
where $E_{0}$ and $E_{1}$ are $\lambda$-twisted Hilbert bundles on a good
covering $\mathcal{U}$ and $D$ is a family of Fredholm operators\footnote{We
note that $\mathrm{HOM}(E,F)$ is an ordinary bundle with fibre $Hom(H,H)=$
$\mathcal{L}(H).$ The space of Fredholm operators "from $E$ to $F"$ is the
subspace of the sections of $\mathrm{HOM}(E,F),$ which are Fredholm over each
point of $X.$} from $E_{0}$ to $E_{1}.$ With the operation induced by the
direct sum of triples, we get a group denoted by $K_{\lambda}(\mathcal{U}).$
We note that $K_{\lambda}(\mathcal{U})$ is a module over $K(\mathcal{U}). $
Here $K(\mathcal{U})$ is a short notation for the usual $K$-theory of the
nerve of $\mathcal{U}.$ If $\mathcal{U}$ is good as in Remark \ref{fine}, it
is isomorphic to the classical topological $K$-group $K(X).$

In order to prove that this last definition is consistent with the previous
ones, we consider the Banach category of $\lambda$-twisted Hilbert bundles. It
is equivalent to the category of bundles of \underline{$\mathcal{L}$%
}$_{\lambda}$-modules, where \underline{$\mathcal{L}$}$_{\lambda}$ is the
algebra bundle above with fibre $\mathcal{L}(H)$ twisted by $\lambda$. Let
\underline{$\mathcal{L}$}$_{\lambda}/\underline{\mathcal{K}}_{\lambda}$ be the
quotient bundle with fibre the Calkin algebra $\mathcal{L}(H)/\mathcal{K}(H).$

\begin{lemma}
Let $\overline{D}$ be the class of $D$ as a morphism between the associated
\underline{$\mathcal{L}$}$_{\lambda}/\underline{\mathcal{K}}_{\lambda}%
$-modules. Then two triples $(E_{0},E_{1},D)$ and $(E_{0}^{\prime}%
,E_{1}^{\prime},D^{\prime})$ are homotopic if and only if the associated
triples $(E_{0},E_{1},\overline{D})$ and $(E_{0}^{\prime},E_{1}^{\prime
},\overline{D}^{\prime})$ are homotopic.
\end{lemma}

\begin{proof}
In general, let us denote also by $\overline{M}$ the class of $M$ as an
\underline{$\mathcal{L}$}$_{\lambda}/\underline{\mathcal{K}}_{\lambda}%
$-module. We have a continuous map%
\[
\mathcal{F}(E_{0},E_{1})\rightarrow\mathrm{Iso}(\overline{E}_{0},\overline
{E}_{1}),
\]
where the notation $\mathcal{F}$ stands for continuous families of Fredholm
maps. According to a classical theorem on Banach spaces, this map admits a
continuous section. Therefore, we get a trivial fibration with contractible
fibre which is the Banach space of sections of the bundle $\underline
{\mathcal{K}}_{\lambda}$. The proposition follows immediately.
\end{proof}

The philosophy of the lemma is that our third definition of twisted $K$-theory
is equivalent to the Grothendieck group of the Banach functor%
\[
\varphi:\mathcal{P}^{\prime}\text{(}\underline{\mathcal{L}}_{\lambda
}\mathcal{)\rightarrow P}^{\prime}\text{(}\underline{\mathcal{L}}_{\lambda
}\text{/}\underline{\mathcal{K}}_{\lambda}\mathcal{)},
\]
as defined in \cite[Section II]{Karoubi livre}. Here the category
$\mathcal{P}^{\prime}$(\underline{$\mathcal{L}$}$_{\lambda}\mathcal{)}$ (resp.
$\mathcal{P}^{\prime}$(\underline{$\mathcal{L}$}$_{\lambda}$/\underline
{$\mathcal{K}$}$_{\lambda}\mathcal{))}$ is equivalent to the category of free
modules over \underline{$\mathcal{L}$}$_{\lambda}$ (resp. \underline
{$\mathcal{L}$}$_{\lambda}/\underline{\mathcal{K}}_{\lambda}\mathcal{)}.$
Since $K_{0}(\mathcal{P}^{\prime}(\underline{\mathcal{L}}_{\lambda}%
\mathcal{)})=0,$ this Grothendieck group is canonically isomorphic to
$K_{0}(\underline{\mathcal{K}}_{\lambda}\mathcal{)}$ which is precisely our
first definition since, as already mentioned, \underline{$\mathcal{K}$%
}$_{\lambda}$ is the algebra bundle with fibre $\mathcal{K}(H)$ associated to
the cocycle $\lambda.$

\begin{remark}
Instead of the Grothendieck group of the functor $\varphi,$ we could as well
consider the group $K_{1}(\underline{\mathcal{L}}_{\lambda}$/$\underline
{\mathcal{K}}_{\lambda})$ which is isomorphic to $K(\varphi),$ since
$\underline{\mathcal{L}}_{\lambda}$ is a flabby ring. We shall use this
equivalent description of twisted $K$-theory in Appendix 8.2.
\end{remark}

\begin{Remarks}
If $\lambda$ is of finite order, the Fredholm definition of twisted $K$-theory
is detailed in \cite[pg. 18]{DK}. If $\lambda=1,$ we recover the theorem of
Atiyah and Janich \cite{Atiyah Book}, \cite{Janich}, in a slightly weaker form.
\end{Remarks}

\medskip

As it is shown in \cite{DK} and \cite{Karoubi}, there is a $\mathbb{Z}%
$/2-graded version of twisted $K$-theory. This version is needed for the Thom
isomorphism in the general case of an arbitrary real vector bundle $V$ (which
is not necessarily oriented). It is also needed for the Poincar\'{e} pairing
applied to arbitrary manifolds. We shall concentrate on the case of non
torsion classes $\left[  \lambda\right]  $ in the third cohomology group of
$X$. The case when $\left[  \lambda\right]  $ is a torsion class in
$H^{3}(X;\mathbb{Z})$ has been extensively studied in \cite{DK}.

The essential idea is to replace the previous structural group $U(H)$ by the
group $\Gamma(H)$ of matrices in $U(H\oplus H)$ of type%
\[
\left(
\begin{array}
[c]{cc}%
g_{1} & 0\\
0 & g_{2}%
\end{array}
\right)
\]
or%
\[
\left(
\begin{array}
[c]{cc}%
0 & h_{1}\\
h_{2} & 0
\end{array}
\right)  .
\]
The point here is that $\Gamma(H)$ acts by inner automorphisms on
$\mathcal{L}(H\oplus H)$ with a degree shift which is either $0$ or $1,$ the
first copy of $H$ being of degree 0 and the second one of degree $1.$ As in
the previous Section, we may give a $\mathbb{Z}/2$-graded module
interpretation of twisted Hilbert bundles modelled on $\Gamma(H)$. If $E$ is
such a graded twisted Hilbert bundle, $A$ $=$ $\mathrm{END}(E)$ is a bundle of
graded algebras with fibre $\mathcal{L}(H\oplus H).$ Conversely, for any
bundle of graded algebras $A$ with fibre $\mathcal{L}(H\oplus H),$ there is a
twisted Hilbert bundle $E$ with structural group $\Gamma(H)$ such that $A$ is
isomorphic to $\mathrm{END}(E).$ According to \cite{DK}, \cite{Karoubi} and
our previous computations, these graded algebras are classified by the
following cohomology group%
\[
H^{1}(X;\mathbb{Z}/2)\times H^{3}(X;\mathbb{Z}),
\]
with a twisted addition rule, as explained in \cite[p. 10]{DK}. The first
invariant in $H^{1}(X;\mathbb{Z}/2)$ is induced by the map
\[
\Gamma(H)\rightarrow\mathbb{Z}/2,
\]
which describes the type of matrices in $\Gamma(H)$ (diagonal or
antidiagonal). The second invariant is defined as before for the underlying
ungraded twisted Hilbert bundle.

If we consider the graded tensor product of the twisted Hilbert bundle $E$ by
the Clifford algebra $C^{0,1}=\mathbb{C}\left[  x\right]  /(x^{2}-1),$ we get
another type of structural group we might call $\Gamma_{1}(H)$ which is simply
$U(H)\times U(H).$ The elements of degree $0$ are of type $(g,g),$ while the
ones of degree $1$ are of type $(g,-g).$ Algebraicallly, this reflects the
fact that over the complex numbers there are two types of $\mathbb{Z}%
/2$-graded Azumaya algebra, up to graded Morita equivalence, which are
$\mathbb{C}$ and $\mathbb{C}\times\mathbb{C}\mathrm{.}$ For simplicity's sake,
in the following discussion, we shall restrict ourselves to the first case
which is the group $\Gamma(H)$ above. We note however that for real graded
vector bundles, there are eight types of graded algebras (up to graded Morita
equivalence) to consider instead of two$,$ as noticed in \cite{DK}. They
correspond to the Clifford algebras $C^{0,n}$ for $n=0,1...,7,$ over the real numbers.

If $E$ and $F$ are two graded twisted Hilbert bundles of structural group
$\Gamma(H)$, a morphism $(g_{i})$ is of degree $0$ (resp. $1)$ if it is
represented locally by a matrix of type%
\[
\left(
\begin{array}
[c]{cc}%
u_{i} & 0\\
0 & v_{i}%
\end{array}
\right)  \qquad\text{ resp.\qquad}\left(
\begin{array}
[c]{cc}%
0 & u_{i}\\
v_{i} & 0
\end{array}
\right)  .
\]
From the previous category equivalences and the definitions in \cite{Karoubi},
we deduce the following theorem.

\begin{theorem}
Let $\lambda$ be a graded twist defined by two cocycles, with classes in
$H^{1}(X;\mathbb{Z}/2)$ and $H^{3}(X;\mathbb{Z})$ respectively. We consider
the set of homotopy classes of couples $(E,\nabla),$ where $E$ is a $\lambda
$-twisted graded Hilbert bundle and $\nabla$ a family of self-adjoint Fredholm
operators on $E$ which are of degree one. With the operation given by the
direct sum of couples, the group obtained is isomorphic to the $\lambda
$-twisted graded $K$-theory defined in \cite{Karoubi}.
\end{theorem}

\begin{remark}
One should point out that there is a variant of this Fredholm definition of
twisted $K$-theory on a base $X$ which is locally compact: the family of
Fredholm operators $\nabla$ must be an isomorphism outside a compact set (see
e.g. \cite{Atiyah Book} or \cite{Karoubi livre}). This remark will be
important for the definition of the Thom isomorphism in Section $6$.
\end{remark}

\begin{remark}
Whatever definition of graded or ungraded twisted $K$-theory we choose, the
group we obtain, denoted by $K_{\lambda}(X)$ in all cases, may be
\textquotedblleft derived\textquotedblright. One nice way to see this is to
notice that we are considering a $K$-group of special Banach algebras (or
$\mathbb{Z}$/2-graded Banach algebras, see \cite{Karoubi}), for instance
$A=\underline{\mathcal{K}}_{\lambda}.$ We then define $K_{\lambda}^{-n}(X)$ as
$K_{n}(A).$ By Bott periodicity for complex Banach algebras, we have
$K_{\lambda}^{-n}(X)\cong K_{\lambda}^{-n-2}(X).$ According to general
theorems on $K$-theory, one shows that%
\[
K_{\lambda}^{-n}(X)\cong Coker(K_{\lambda}(X)\rightarrow K_{\pi^{\ast}\lambda
}(X\times S^{n})),
\]
where $\pi:X\times S^{n}\rightarrow X$ is the canonical projection. We note
here that the smash product $X\wedge S^{n}$ cannot be used to define
$K_{\lambda}^{-n}(X)$, since there is no associated twist in the cohomology of
$X\wedge S^{n}$ in general.
\end{remark}

As a consequence, we may apply Mayer-Vietoris arguments to the direct sum
$K_{\lambda}(X)\oplus K_{\lambda}^{-1}(X),$ as for the $K$-theory of general
Banach algebras.

\section{Multiplicative structures}

Since we have defined twisted $K$-theory in three ways (at least in the non
graded case), we should investigate the possible multiplicative structure from
these different viewpoints and show that they coincide up to isomorphism.
These multiplicative structures were also investigated in a more general
framework in \cite{Laurent Tu Xu}.

The end result is a \textquotedblleft cup-product\textquotedblright%
\[
K_{\lambda}(X)\times K_{\mu}(X)\rightarrow K_{\lambda\mu}(X),
\]
where $\lambda$ and $\mu$ are two $2$-cocycles\footnote{See Appendix 8.3 for a
possible pairing if we replace $\lambda$ and $\mu$ by their cohomology classes
in $H^{2}(X;S^{1})\cong H^{3}(X;\mathbb{Z}).$} with values in $S^{1}$. Since
$K_{\lambda}(X)$ is the $K$-theory of the Banach algebra \underline
{$\mathcal{K}$}$_{\lambda}$ in general, it is enough to define a continuous
bilinear pairing between nonunital Banach algebras%
\[
\varphi:\underline{\mathcal{K}}_{\lambda}\times\underline{\mathcal{K}}_{\mu
}\rightarrow\mathcal{K}_{\lambda\mu},
\]
such that $\varphi(aa^{\prime},bb^{\prime})=\varphi(a,b)\varphi(a^{\prime
},b^{\prime}).$ The implication that such a $\varphi$ induces a pairing
between $K$-groups is not completely obvious and relies on excision in $K$-theory.

To define the pairing $\varphi$, we observe that if $E_{\lambda}$ is a twisted
Hilbert bundle with twist $\lambda$ and $F_{\mu}$ another one with twist
$\mu,E\widehat{\otimes}F$ is a twisted Hilbert bundle with twist $\lambda\mu.$
Here, the fibres of $E_{\lambda}\widehat{\otimes}F_{\mu}$ are the Hilbert
tensor product of the fibres of $E$ and $F$ respectively (we implicitly
identify the Hilbert tensor product of $H\otimes H$ with $H$ since it is
infinite dimensional). Therefore, we have a pairing between Banach bundles%
\[
\mathrm{END}(E_{\lambda})\times\mathrm{END}(F_{\mu})\rightarrow\mathrm{END}%
(E_{\lambda}\widehat{\otimes}F_{\mu}),
\]
which is bilinear and continuous. If we take continuous sections, we deduce
the map $\varphi$ required. We note that $\varphi$ also induces a ring map%
\[
\underline{\mathcal{K}}_{\lambda}\widehat{\otimes}\underline{\mathcal{K}}%
_{\mu}\rightarrow\mathcal{K}_{\lambda\mu}.
\]
The symbol $\widehat{\otimes}$ now denotes the completed projective tensor
product of Grothendieck. The inclusion of Banach algebras
\[
\underline{\mathcal{K}}_{\lambda}\widehat{\otimes}\underline{\mathcal{K}}%
_{\mu}\subset\mathcal{K}_{\lambda\mu}%
\]
is not an isomorphism. However, when $X$ varies, both functors define a
(twisted) cohomology theory which is the usual $K$-theory when $X$ is
contractible. Therefore, by a standard Mayer-Vietoris argument and Bott
periodicity, this inclusion induces an isomorphism on $K$-groups.

We should note that this cup-product is much simpler to define if $\left[
\lambda\right]  $ and $\left[  \mu\right]  $ are torsion classes in the
cohomology. According to Section $3,$ we may then assume that $E$ and $F$ are
finite dimensional twisted vector bundles. The cup-product is simply the usual
one\footnote{As often, we underline the algebra of sections of the algebra
bundles involved.}%
\[
K(\underline{A})\times K(\underline{B})\rightarrow K(\underline{A\otimes B})
\]
where $A=\mathrm{END}(E_{\lambda})$ and $B=\mathrm{END}(F_{\mu})$ are bundles
of finite dimensional algebras, with matrix algebras as fibres.

Coming back to the general case, we now use our second definition of twisted
$K$-theory in order to get a cup-product between $K$-groups of unital rings.
According to Section $4,$ we have exact sequences of Banach algebras%
\begin{align*}
0  &  \rightarrow\underline{\mathcal{K}}_{\lambda}\rightarrow\underline
{\mathcal{A}}_{\lambda}\rightarrow\underline{\mathcal{B}}_{\lambda}%
\rightarrow0\\
0  &  \rightarrow\underline{\mathcal{K}}_{\mu}\rightarrow\underline
{\mathcal{A}}_{\mu}\rightarrow\underline{\mathcal{B}}_{\mu}\rightarrow0,
\end{align*}
which split as exact sequences of Banach spaces. Therefore, we deduce another
exact sequence by taking completed projective tensor products of Banach
algebras%
\[
0\rightarrow\underline{\mathcal{K}}_{\lambda}\widehat{\otimes}\underline
{\mathcal{K}}_{\mu}\rightarrow\underline{\mathcal{A}}_{\lambda}\widehat
{\otimes}\underline{\mathcal{A}}_{\mu}\rightarrow\mathcal{D}_{\lambda,\mu
}\rightarrow0.
\]
The Banach algebra $\mathcal{D}_{\lambda,\mu}$ is the following fibre product%
\[%
\begin{array}
[c]{ccc}%
\mathcal{D}_{\lambda,\mu} & \rightarrow & \underline{\mathcal{A}}_{\lambda
}\widehat{\otimes}\underline{\mathcal{B}}_{\mu}\\
\downarrow &  & \downarrow\\
\underline{\mathcal{B}}_{\lambda}\widehat{\otimes}\underline{\mathcal{A}} &
\rightarrow & \underline{\mathcal{B}}_{\lambda}\widehat{\otimes}%
\underline{\mathcal{B}}_{\mu}%
\end{array}
\]
Since the algebras $\underline{\mathcal{B}}_{\lambda}$ and $\underline
{\mathcal{B}}_{\mu}$ are flabby, the algebra $\mathcal{D}_{\lambda,\mu}$ has
trivial $K$-groups. It follows that the map%
\[
\underline{\mathcal{K}}_{\lambda}\widehat{\otimes}\underline{\mathcal{K}}%
_{\mu}\rightarrow\underline{\mathcal{A}}_{\lambda}\widehat{\otimes}%
\underline{\mathcal{A}}_{\mu}%
\]
is a $K$-theory equivalence. Therefore, we may also define a cup-product%
\[
K(\underline{\mathcal{A}}_{\lambda})\times K(\underline{\mathcal{A}}_{\mu
})\rightarrow K(\underline{\mathcal{A}}_{\lambda\mu}),
\]
as the following composition%
\[
K(\underline{\mathcal{A}}_{\lambda})\times K(\underline{\mathcal{A}}_{\mu
})\rightarrow K(\underline{\mathcal{A}}_{\lambda}\widehat{\otimes}%
\underline{\mathcal{A}}_{\mu})\cong K(\underline{\mathcal{K}}_{\lambda
}\widehat{\otimes}\underline{\mathcal{K}}_{\mu})\overset{\cong}{\rightarrow
}K(\underline{\mathcal{K}}_{\lambda\mu})\cong K(\underline{\mathcal{A}%
}_{\lambda\mu}).
\]
If we identity $K(\underline{\mathcal{A}}_{\lambda}\widehat{\otimes}%
\underline{\mathcal{A}}_{\mu})$ with $K(\underline{\mathcal{A}}_{\lambda\mu})
$ by the previous sequence of isomorphisms, we may take as the definition for
our cup-product the pairing%
\[
K(\underline{\mathcal{A}}_{\lambda})\times K(\underline{\mathcal{A}}_{\mu
})\rightarrow K(\underline{\mathcal{A}}_{\lambda}\widehat{\otimes}%
\underline{\mathcal{A}}_{\mu}).
\]

We now come to the third definition of the cup-product in terms of Fredholm
operators. As is well known (see e.g. \cite{Atiyah Book}, \cite{DK} or
\cite{Karoubi}), one advantage of this definition of twisted $K$-theory (for
$\left[  \lambda\right]  $ of finite or infinite order) is a handy description
of the cup-product. In the ungraded case, it is more convenient still to view
$E=E_{1}\oplus E_{1}$ as a $\mathbb{Z}/2$-graded twisted bundle and
replace\footnote{More correctly, we should write $D$ as a section of the
bundle $\mathrm{HOM}(E_{0},E_{1}).$} $D:E_{0}\rightarrow E_{1}$ by the
following operator $\nabla$ which is self-adjoint and of degree 1:%
\[
\nabla=\left(
\begin{array}
[c]{cc}%
0 & D^{\ast}\\
D & 0
\end{array}
\right)  .
\]
The cup-product of $(E,\nabla)$ with another couple of the same type
$(E^{\prime},\nabla^{\prime})$ is simply defined by the formula%
\[
(E,\nabla)\smallsmile(E^{\prime},\nabla^{\prime})=(E\widehat{\otimes}%
E^{\prime},\nabla\widehat{\otimes}1+1\widehat{\otimes}\nabla^{\prime}).
\]
Here the symbol $\widehat{\otimes}$ denotes the graded and Hilbert tensor
product. We notice that if $E$ is associated to the twist $\lambda,$
$E^{\prime}$ to the twist $\lambda^{\prime},$ the cup-product is associated to
the twist $\lambda\cdot\lambda^{\prime}$, a cocycle whose cohomology class is
the sum of the two related cohomology classes in $H^{2}(X;S^{1})$.

It is not completely obvious that this third definition of the cup-product is
equivalent to the previous one with the bundles \underline{$\mathcal{K}$%
}$_{\lambda}$ or \underline{$\mathcal{A}$}$_{\lambda}$. In order to prove this
technical point, we use the results of Appendix 8.2 describing explicitly the
isomorphism between $K(\underline{\mathcal{K}}_{\lambda})$ and $K_{1}%
(\underline{\mathcal{B}}_{\lambda}/\underline{\mathcal{K}}_{\lambda}).$ In
fact, any element of $K_{1}(\underline{\mathcal{B}}_{\lambda}/\underline
{\mathcal{K}}_{\lambda})$ is the cup-product of an element $u$ of
$K(\underline{\mathcal{K}}_{\lambda})$ by a generator $\tau$ of
\[
K_{1}(\mathcal{L}/\mathcal{K})\cong\mathbb{Z}.
\]
This generator is classically defined by the shift (as a Fredholm operator).
Moreover, we may assume that $u$ is induced by a self-adjoint involution on
$M_{2}((\underline{\mathcal{K}}_{\lambda})^{+}),$ where $(\underline
{\mathcal{K}}_{\lambda})^{+}$ is the algebra $\underline{\mathcal{K}}%
_{\lambda}$ with a unit added. On the other hand, both $K_{\ast}%
(\underline{\mathcal{K}}_{\lambda})$ and $K_{1+\ast}(\underline{\mathcal{B}%
}_{\lambda}/\underline{\mathcal{K}}_{\lambda})$ may be considered as (twisted)
cohomology theories on $X.$ Therefore, again by a Mayer-Vietoris argument, the
formula above for the cup-product with Fredholm operators has to be compared
with the previous one only when $X$ is reduced to a point, a case which is obvious.

\bigskip

This Fredholm multiplicative setting has the advantage that it may be extended
to the graded version of twisted $K$-theory by the same formula%
\[
(E,\nabla)\smallsmile(E^{\prime},\nabla^{\prime})=(E\widehat{\otimes}%
E^{\prime},\nabla\widehat{\otimes}1+1\widehat{\otimes}\nabla^{\prime}).
\]
If $(\left[  \lambda_{1}\right]  ,\left[  \lambda_{3}\right]  )$ and $(\left[
\lambda_{1}^{\prime}\right]  ,\left[  \lambda_{3}^{\prime}\right]  )$ are the
twists of $E$ and $E^{\prime}$ respectively, the twist of $E\widehat{\otimes
}E^{\prime}$ in cohomology is ($\left[  \mu_{1}\right]  ,\left[  \mu
_{3}\right]  ),$ where%
\[
\left[  \mu_{1}\right]  =\left[  \lambda_{1}\right]  +\left[  \lambda
_{1}^{\prime}\right]
\]
and%
\[
\left[  \mu_{3}\right]  =\left[  \lambda_{3}\right]  +\left[  \lambda
_{3}^{\prime}\right]  +\beta(\left[  \lambda_{1}\right]  \cdot\left[
\lambda_{1}^{\prime}\right]  ).
\]
Here $\beta:H^{2}(X;\mathbb{Z}/2)\rightarrow H^{3}(X;\mathbb{Z})$ is the
Bockstein homomorphism (compare with \cite[p. 10]{DK}). Thanks to the Thom
isomorphism which is proved in \cite{Karoubi} (see also the next section and
\cite{Carey-Wang}), this graded cup-product is compatible with the ungraded
one defined on the Thom space of the orientation bundle determined by the
graded twist.

\section{Thom isomorphism and operations in twisted $K$-theory}

This Section is just a short rewriting of the corresponding sections 4 and 7
of \cite{Karoubi}, with the point of view of twisted Hilbert bundles. It is
added here for completeness' sake.

\medskip

In order to define the Thom isomorphism in twisted $K$-theory, as in
\cite{Karoubi} and \cite{Carey-Wang} with our new point of view, we need to
consider twisted Hilbert bundles $E$ with a Clifford module structure. Such a
structure is given by a finite dimensional real vector bundle $V$ on $X,$
provided with a positive metric $q$ and an action of $V$ on $E,$ such that
$(v)^{2}=q(v).1$. Now let $\lambda$ be a graded twist, given by a covering
$\mathcal{U}\mathbb{\ }$= $(U_{i})$ together with a couple $(\lambda
_{1},\lambda_{3})$ consisting of a $1$-cocycle with values in $\mathbb{Z}/2$
and a $2$-cocycle with values in $S^{1}.$ We define the Grothendieck group
$K_{\lambda}^{V}(X)$ from the set of homotopy classes of couples%
\[
(E,\nabla),
\]
as follows: $E$ is a $\mathbb{Z}/2$-graded twisted Hilbert bundle which is
also a graded $C(V)$-module, $V$ acting by self-adjoint endomorphisms of
degree $1$. Moreover, the family of Fredholm operators $\nabla$ must satisfy
the following properties

1) $\nabla$ is self-adjoint and of degree $1,$ as in the previous section,

2) $\nabla$ anticommutes with the elements $v$ in $V.$

This group is not entirely new. Using our dictionary relating twisted Hilbert
bundles and module bundles, we described it in great detail in \cite[\S \ 4]%
{Karoubi}$.$ We should also notice that this structure of $C(V)$-module may be
integrated into the twist $\lambda$: if $w_{1}=w_{1}(V)$ and $w_{2}=w_{2}(V)$
are the first two Stiefel-Whitney classes of $V,$ one has to replace $\lambda$
by the sum of $\lambda$ and $C(V)$ in the graded Brauer group (this was one of
the main motivations for the paper \cite{DK}). More precisely, the resulting
cohomology classes are%
\[
\left[  \lambda_{1}\right]  +w_{1}(V)
\]
in degree one and%
\[
\left[  \lambda_{3}\right]  +\beta(\left[  \lambda_{1}\right]  \cdot
w_{1})+\beta(w_{2})
\]
in degree $3.$

\bigskip

Using our previous reference \cite{Karoubi}, we are now able to define the
Thom isomorphism%

\[
t:K_{\lambda}^{V}(X)\rightarrow K_{\pi^{\ast}\lambda}(V)
\]
in simpler terms. If $\pi$ denotes the projection $V\rightarrow X,$ and if
$(E,\nabla)$ defines an element of the group $K_{\lambda}^{V}(X),$ we define
$t(E,\nabla)$ as the couple $(\pi^{\ast}(E),\nabla^{\prime}),$ where
$\nabla^{\prime}$ is defined over a point $v$ of $V$, with projection $x,$ by
the formula%
\[
\nabla_{v}^{\prime}=v+\nabla_{x}.
\]
We recognize here the formula already given in \cite{Karoubi}: we have just
replaced module bundles by twisted Hilbert bundles.

\bigskip\bigskip

Operations on twisted $K$-theory have already been defined in many references
\cite{DK}, \cite{Atiyah and Segal}, \cite{Karoubi}. Twisted Hilbert bundles
give a nice framework to redefine them. For simplicity's sake, we restrict
ourselves to ungraded twisted $K$-groups.

If we start with an element $(E,\nabla)$ defining an element of $K_{\lambda
}(X)$ as at the end of Section 4, its $k^{th}$-power\footnote{where the symbol
$\widehat{\otimes}$ denotes again the graded Hilbert tensor product.}%
\[
(E^{\widehat{\otimes}k},\nabla\widehat{\otimes}...\widehat{\otimes
}1+...+1\widehat{\otimes}...\widehat{\otimes}\nabla)
\]
has an obvious action of the symmetric group $S_{k}.$ We should notice that
the twist of the $k^{th}$-power is $\lambda^{k}.$ According to Atiyah's
philosophy \cite{Atiyah Book}, the $k^{th}$-power defines a map%
\[
K_{\lambda}(X)\rightarrow K_{\lambda^{k}}(X)\otimes_{\mathbb{Z}}R(S_{k}),
\]
where $R(S_{k})$ denotes the complex representation ring of $S_{k}.$
Therefore, any $\mathbb{Z}$\textrm{-}homomorphism%
\[
R(S_{k})\rightarrow\mathbb{Z}%
\]
gives rise to an operation in twisted $K$-theory. In particular, the
Grothendieck exterior powers and the Adams operations may be defined in
twisted $K$-theory, using Atiyah's method.

As an interesting $\mathbb{Z}$-homomorphism from $R(S_{k})$ to $\mathbb{Z},$
one may choose the map which associates to a complex representation $\rho$ the
trace of $\rho(c_{k}),$ where $c_{k}$ is the cycle $(1,2,...,k),$ a trace
which is in fact an integer. The resulting homomorphism%
\[
K_{\lambda}(X)\rightarrow K_{\lambda^{k}}(X)
\]
is quite explicit. We associate to $F=(E,\nabla)$ the Gauss sum
\[%
{\displaystyle\sum}
(F^{\otimes k})_{n}\otimes\omega^{n}%
\]
in the group $K_{\lambda^{k}}(X)\otimes_{\mathbb{Z}}\Omega_{k}$, where
$\Omega_{k}$ is the ring of $k$-cyclotomic integers. In this sum, $\omega$ is
a primitive $k^{th}$-root of unity. The element $(F^{\otimes k})_{n}$ is the
eigenmodule associated to the eigenvalue $\omega^{n}$ of a generator of the
cyclic group $C_{k}$ acting on $F^{\otimes k}$. This sum belongs in fact to
$K_{\lambda^{k}}(X),$ as a subgroup of $K_{\lambda^{k}}(X)\otimes\Omega_{k}.$
As shown by Atiyah \cite{Atiyah Book}, we get this way a nice alternative
definition of the Adams operation $\Psi^{k}.$

\begin{remark}
If the class of $\lambda$ in $H^{3}(X;\mathbb{Z}\mathrm{)}$ is of finite
order, it is not necessary to consider twisted Hilbert bundles and Fredholm
operators. One just deal with finite dimensional twisted vector bundles as in
Section $3.$
\end{remark}

\begin{remark}
One should notice that operations are much more delicate to define in graded
twisted $K$-theory, even for coefficients $\left[  \lambda\right]  $ of finite
order in $H^{3}(X;\mathbb{Z})$. This was pointed out in \cite{DK} and recalled
in \cite{Karoubi}. Fredholm operators were already introduced in \cite{DK} in
order to deal with this problem, before subsequent works on twisted $K$-theory.
\end{remark}

\section{Connections and the Chern homomorphism}

Let us now assume that $X$ is a manifold. The previous definitions make sense
in the differential category. The fact that we get the same $K$-groups is more
or less standard and relies on arguments going back to Steenrod
\cite{Steenrod}. As an illustrative example, the \v{C}ech cohomologies
$H^{1}(X;GL_{n}(\mathbb{C}))$ and $H^{1}(X;PGL_{n}(\mathbb{C}))$ may be
computed with differential cochains. Therefore the classification of
topological algebra bundles (with fibre $M_{n}(C))$ is the same in the
differential category. The same general result is true for module bundles and
therefore for twisted $K$-theory, if we choose differential $2$-cocycles
$\lambda$ with values in $S^{1}$ to parametrize the twisted $K$-groups.

In the differential category, the definition of the Chern homomorphism between
twisted $K$-theory and \textquotedblleft twisted cohomology\textquotedblright%
\ was given in many papers \cite{Atiyah and Segal 2}, \cite{Mathai Stevenson},
\cite{Tsygan}, \cite{Tu et Xu}, \cite{Carey-Wang2}, and \cite{Gorokhovsky}.
Our method is more elementary and is based on the classical definitions of
Chern-Weil theory applied to twisted bundles\footnote{For the classical
computations, we refer to the books \cite[pg. 78]{Kobayashi-Nomizy} and
\cite{Karoubi CRAS} for instance.}. We start with twisted finite dimensional
bundles which are easier to handle. However, as we shall see later on, the
same method may be applied to infinite dimensional bundles in the spirit of
Section 4.

Let $E$ be a twisted vector bundle of rank\footnote{The rank may vary above
different connected component of $X.$} $n,$ defined on a covering
$\mathcal{U}$ $=(U_{i})$ by transition functions $(g_{ji}),$ with the twisted
cocycle condition%
\[
g_{ki}=g_{kj}\cdot g_{ji}\cdot\lambda_{kji},
\]
as in Section $3.$ We assume that all functions are of class $C^{\infty},$
which does not change the classification problem for twisted bundles as we
have seen previously.

\begin{defn}
A connection $\Gamma$ on $E$ is given by $(n\times n)$-matrices $\Gamma_{i}$
of $1$-differential forms on $U_{i}$ such that on $U_{i}\cap U_{j}$ we have
the relation%
\[
\Gamma_{i}=g_{ji}^{-1}\cdot\Gamma_{j}\cdot g_{ji}+g_{ji}^{-1}\cdot
dg_{ji}+\omega_{ji}.1.
\]
Here $\omega_{ji}$ is a $1$-differential form related to the $\lambda_{kji}$
by the following relation%
\[
\omega_{ji}-\omega_{ki}+\omega_{kj}=\lambda_{kji}^{-1}\cdot d\lambda_{kji}.
\]

\end{defn}

\bigskip

Moreover, from the relation above with the $\Gamma^{\prime}$s, we deduce that
$\omega_{ij}=-\omega_{ji}.$ If we take the differential of the previous
relation, we also get%
\[
d\omega_{ji}-d\omega_{ki}+d\omega_{kj}=0.
\]
In the applications below, $\omega$ will be a differential form with values in
$i\mathbb{R},$ where\footnote{The two different meanings of the symbol $"i"$
are clear from the context.} $i=\sqrt{-1}$ (if the $g_{ji}$ are unitary operators).

\bigskip

\begin{ex}
(which shows the existence of such connections). Let $(\alpha_{k})$ be a
partition of unity associated to the covering $\mathcal{U}$. We then consider
the \textquotedblleft barycentric connection\textquotedblright\ defined by the
formula%
\[
\Gamma_{i}=\sum_{k}\alpha_{k}\cdot g_{ki}^{-1}\cdot dg_{ki}.
\]
Since $g_{ki}=g_{kj}\cdot g_{ji}\cdot\lambda_{kji},$ we have the following
expansion%
\[
g_{ki}^{-1}\cdot dg_{ki}=g_{ji}^{-1}\cdot(g_{kj}^{-1}\cdot dg_{kj})\cdot
g_{ji}+g_{ji}^{-1}\cdot dg_{ji}+\lambda_{kji}^{-1}\cdot d\lambda_{kji}.
\]
Therefore, on $U_{i}\cap U_{j}$ we have the expected identity%
\[
\Gamma_{i}=g_{ji}^{-1}\cdot\Gamma_{j}\cdot g_{ji}+g_{ji}^{-1}\cdot
dg_{ji}+\omega_{ji}\cdot1,
\]
where%
\[
\omega_{ji}=\sum_{k}\alpha_{k}\cdot\lambda_{kji}^{-1}\cdot d\lambda_{kji}.
\]

\end{ex}

\bigskip

\begin{remark}
It is clear from the definition that the space of connections on $E$ is an
affine space: if $\Gamma$ and $\nabla$ are two connections on $E,$ for any
real number $t,$ $(1-t)\Gamma+t\nabla$ is also a connection$.$
\end{remark}

\bigskip\smallskip

We have choosen a definition of a connexion in terms of \textquotedblleft
local coordinates\textquotedblright. However, we have to check how connections
correspond when we change them. In other terms, let $(\alpha)$ be an
isomorphism from the coordinate bundle $(h)$ to $(g)$ as in Section 1.
According to Formula(\ref{morphisms}), we then have the relation%
\[
g_{ji}\cdot\alpha_{i}=\alpha_{j}\cdot h_{ji}%
\]
Associated to this morphism, we define the pull back $\alpha^{\ast}(\Gamma)$
of the connection $(\Gamma)$ as locally defined on the coordinate bundle $(h)
$ by the formula%
\[
\nabla_{i}=\alpha_{i}^{-1}\cdot\Gamma_{i}\cdot\alpha_{i}+\alpha_{i}^{-1}\cdot
d\alpha_{i}%
\]
In order for this to make sense, we have to check the relation%
\[
\nabla_{i}=h_{ji}^{-1}\cdot\nabla_{j}\cdot h_{ji}+h_{ji}^{-1}\cdot
dh_{ji}+\omega_{ji}.1.
\]
which is slightly tedious. We start from the formula%
\[
\Gamma_{i}=g_{ji}^{-1}\cdot\Gamma_{j}\cdot g_{ji}+g_{ji}^{-1}\cdot
dg_{ji}+\omega_{ji}.1,
\]
where we replace $g_{ji}$ by $\alpha_{j}\cdot h_{ji}\cdot\alpha_{i}^{-1}$. We
also replace $dg_{ji}$ by%
\[
dg_{ji}=d\alpha_{j}\cdot h_{ji}\cdot\alpha_{i}^{-1}+\alpha_{j}\cdot
dh_{ji}\cdot\alpha_{i}^{-1}-\alpha_{j}\cdot h_{ji}\cdot\alpha_{i}^{-1}\cdot
d\alpha_{i}\cdot\alpha_{i}^{-1}.
\]
We then get%
\[
\nabla_{i}=\alpha_{i}^{-1}\cdot(g_{ji}^{-1}\cdot\Gamma_{j}\cdot g_{ji}%
+g_{ji}^{-1}\cdot dg_{ji}+\omega_{ji}.1.)\cdot\alpha_{i}+\alpha_{i}^{-1}\cdot
d\alpha_{i}%
\]%
\[
=\alpha_{i}^{-1}\cdot(g_{ji}^{-1}\cdot\Gamma_{j}\cdot g_{ji})\cdot\alpha_{i}%
\]%
\[
+\alpha_{i}^{-1}\cdot g_{ji}^{-1}\cdot(d\alpha_{j}\cdot h_{ji}\cdot\alpha
_{i}^{-1}+\alpha_{j}\cdot dh_{ji}\cdot\alpha_{i}^{-1}-\alpha_{j}\cdot
h_{ji}\cdot\alpha_{i}^{-1}\cdot d\alpha_{i}\cdot\alpha_{i}^{-1})\cdot
\alpha_{i}%
\]%
\[
+\omega_{ji}.1.
\]%
\[
=h_{ji}^{-1}\cdot(\nabla_{j}-\alpha_{j}^{-1}\cdot d\alpha_{j})\cdot h_{ji}%
\]%
\[
+h_{ji}^{-1}\cdot\alpha_{j}^{-1}\cdot d\alpha_{j}\cdot h_{ji}+h_{ji}^{-1}\cdot
dh_{ji}-\alpha_{i}^{-1}\cdot d\alpha_{i}+\alpha_{i}^{-1}\cdot d\alpha
_{i}+\omega_{ji}.1.
\]%
\[
=h_{ji}^{-1}\cdot\nabla_{j}\cdot h_{ji}+h_{ji}^{-1}\cdot dh_{ji}+\omega
_{ji}.1,
\]
which is the expected formula.

The \textquotedblleft local curvatures\textquotedblright\ $R_{i}$ associated
to the $\Gamma_{i}$ are given by the usual formula\footnote{As in classical
Chern-Weil theory, one may also write $1/2\left[  \Gamma_{i},\Gamma
_{i}\right]  $ instead of $(\Gamma_{i})^{2}.$}%
\[
R_{i}=d\Gamma_{i}+(\Gamma_{i})^{2}.
\]
Unfortunately, the traces of these local curvatures do not agree on $U_{i}\cap
U_{j}$, since a simple computation as above leads to the relation%
\[
R_{i}=g_{ji}^{-1}\cdot R_{j}\cdot g_{ji}+d\omega_{ji}.1.
\]
However, using a partition of unity $(\alpha_{i}),$ as in the case of the
barycentric connection, we may define a family of \textquotedblleft twisted
curvatures\textquotedblright\ by the following formula, where $m=1,2,...$%
\[
R_{(m)}=\sum_{i}\alpha_{i}\cdot(R_{i})^{m}.
\]
We now define a family of \textquotedblleft Chern characters\textquotedblright%
\ $Ch_{(m)}(E,\Gamma)$ as%
\[
Ch_{(m)}(E,\Gamma)=Tr(R_{(m)}).
\]
We should notice that $Ch_{(m)}(E,\Gamma)$ belongs to the vector space of
differential forms with values in $(i)^{m}\mathbb{R},$ since the $g_{kl}$ are
unitary matrices. By convention, we put%
\[
Ch_{(0)}(E,\Gamma)=n.
\]
The differential of $Ch_{(1)}$ is%
\[
d(Ch_{(1)}(E,\Gamma))=\sum_{i}\alpha_{i}\cdot.Tr(dR_{i})+\sum_{i}d\alpha
_{i}\cdot Tr(R_{i}).
\]
It is well known (and easy to prove) that
\[
Tr(dR_{i})=Tr(d\Gamma_{i}\cdot\Gamma_{i}-\Gamma_{i}\cdot d\Gamma_{i})=0.
\]
On the other hand, the relation between $R_{i}$ and $R_{j}$ above leads to the
following identity between differential forms on $U_{j}:$%
\[
\sum_{i}d\alpha_{i}\cdot Tr(R_{i})=(\sum_{i}d\alpha_{i}\cdot Tr(R_{j}%
))+n\sum_{i}d\alpha_{i}\cdot d\omega_{ji}=n\sum_{i}d\alpha_{i}\cdot
d\omega_{ji}.
\]
The $3$-differential form $\theta_{j}=%
{\displaystyle\sum\limits_{k}}
d\alpha_{k}\cdot d\omega_{jk}$ is clearly closed on $U_{j}.$ Moreover, on
$U_{i}\cap U_{j}$ we have%
\[
\theta_{j}-\theta_{i}=%
{\displaystyle\sum\limits_{k}}
d\alpha_{k}\cdot(d\omega_{jk}-d\omega_{ik})=%
{\displaystyle\sum\limits_{k}}
d\alpha_{k}\cdot d\omega_{ji}=0,
\]
according to the relation above between various $d\omega^{\prime}s.$
Therefore, the $\left\{  \theta_{i}\right\}  $ define a global $3$%
-differential form $\theta$ on the manifold $X$ with values in $i\mathbb{R}$.
This $3$-cohomology class is the opposite of the image of $\lambda$ by the
connecting homomorphism\footnote{See Appendix 8.1 for a proof of this
statement..}$:$%
\[
H^{2}(X;S^{1})\rightarrow H^{3}(X;2i\pi\mathbb{Z}),
\]
associated to the classical exact sequence of sheaves:%
\[
0\rightarrow2i\pi\mathbb{Z}\rightarrow i\mathbb{R}\overset{\exp}%
{\mathbb{\rightarrow}}\mathbb{S}^{1}\rightarrow0,
\]
followed by the map\footnote{Here $\mathbb{R}^{\delta}$ denotes now the field
$\mathbb{R}$ with the discrete topology.}
\[
H^{3}(X;2i\pi\mathbb{Z})\rightarrow H^{3}(X;i\mathbb{R}^{\delta}),
\]
deduced from the inclusion $\mathbb{Z\subset R}^{\delta}.$ Summarizing the
above discussion, we get our first relation%
\[
d(Ch_{(1)}(E,\Gamma))=n\cdot\theta.
\]
Analogous computations can be made with $R_{(2)},R_{(3)},$ etc.. For an
arbitrary $m$,
\[
d(Ch_{(m)}(E,\Gamma))=%
{\displaystyle\sum\limits_{i}}
\alpha_{i}\cdot Tr(d(R_{i})^{m})+%
{\displaystyle\sum\limits_{i}}
d\alpha_{i}\cdot Tr(R_{i})^{m}.
\]
Since $Tr(d(R_{i})^{m})=0$ for the same reasons as above, we have%
\[
d(Ch_{(m)}(E,\Gamma))=%
{\displaystyle\sum\limits_{i}}
d\alpha_{i}\cdot Tr(R_{i})^{m}.
\]
On the other hand, from the relation%

\[
Tr(R_{i})^{m}=Tr(R_{j})^{m}+mTr(R_{j})^{m-1}\cdot d\omega_{ji},
\]
we deduce the following identity between differential forms on $U_{j}:$%

\[%
{\displaystyle\sum\limits_{i}}
d\alpha_{i}\cdot Tr(R_{i})^{m}=%
{\displaystyle\sum\limits_{i}}
d\alpha_{i}\cdot Tr(R_{j})^{m}+%
{\displaystyle\sum\limits_{i}}
m\cdot d\alpha_{i}\cdot Tr(R_{j})^{m-1}\cdot.d\omega_{ji}%
\]%
\[
=m\cdot Tr(R_{j})^{m-1}\cdot\theta.
\]
Therefore,
\[%
{\displaystyle\sum\limits_{i}}
d\alpha_{i}\cdot Tr(R_{i})^{m}=%
{\displaystyle\sum\limits_{j}}
\alpha_{j}%
{\displaystyle\sum\limits_{i}}
d\alpha_{i}\cdot Tr(R_{i})^{m}%
\]%
\[
=%
{\displaystyle\sum\limits_{j}}
m\cdot\alpha_{j}\cdot Tr(R_{j})^{m}\cdot\theta=m\cdot Ch_{(m-1)}%
(E,\Gamma)\cdot\theta.
\]
Summarizing again, we get the relation%

\[
d(Ch_{(m)}(E,\Gamma))=m\cdot Ch_{(m-1)}(E,\Gamma)\cdot\theta.
\]
We now define the total Chern character of $(E,\Gamma)$ with values in the
even de Rham forms\footnote{More precisely, $\Omega^{2k}(X)$ is the vector
space of $2k$-differential forms with values in $(i)^{k}\mathbb{R}$.}
\[
\Omega^{0}(X)\oplus\Omega^{2}(X)\oplus...\oplus\Omega^{2m}(X)\oplus...
\]
by the following formula:%
\[
Ch(E,\Gamma)=Ch_{(0)}(E,\Gamma)+Ch_{(1)}(E,\Gamma)+\frac{1}{2!}Ch_{(2)}%
(E,\Gamma)+...
\]%
\[
+\frac{1}{m!}Ch_{(m)}(E,\Gamma)+...
\]
We have choosen the coefficients in front of the $Ch_{(m)}$ such that
$Ch(E,\Gamma)$ is a cycle in the even/odd de Rham complex with the
differential given by $D=d-.\theta,$ where $.\theta$ is the map defined by the
cup-product with $\theta.$

In Appendix 8.3, we prove by classical considerations that this total Chern
character is well defined as a twisted cohomology class and does not depend on
the connection $\Gamma$ and on the partition of unity. This remark is also
valid in the infinite dimensional case which will be studied later on.

For the time being, since we consider finite dimensional bundles, the class
$\theta$ in $H^{3}(X;i\mathbb{R}^{\delta})$ is reduced to $0.$ Therefore, by
classical considerations on complexes using exponentials of even forms
\cite{Atiyah and Segal 2}, we see that the target of this special Chern
character reduces to the classical one. Moreover, we may also consider twisted
bundles over $X\times S^{1},$ which enables us to define a Chern character
from odd twisted $K$-groups to odd twisted cohomology. From standard
Mayer-Vietoris arguments and Bott periodicity, we deduce that the Chern
character induces an isomorphism between $K_{\lambda}(X)\otimes_{\mathbb{Z}%
}\mathbb{R}$ and $H^{even}(X;\mathbb{R}).$

\medskip

We want to extend the previous considerations to the case when the cohomology
class $\left[  \lambda\right]  $ is of infinite order. For this, we use the
second definition of twisted $K$-theory in terms of twisted principal bundles
associated to the group $G$ of couples $(g,h)$ such that $g $ and $h$ are
invertible operators in $\mathcal{L}(H)$ with $g-h$ compact. However, in order
to be able to take traces, we have to modify slightly this group by assuming
moreover that $g-h$ is a trace class operator, i.e. belongs to $L^{1}.$ By
abuse of notation, we still call $\mathcal{A}\mathbb{\ }$the
algebra\footnote{The norm of an element $(g,h)$ is the sum of the operator
norm on $g$ and the $L^{1}$-norm on $g-h.$} of couples $(g,h)\in
\mathcal{L}\times\mathcal{L}$ such that $g-h\in L^{1}.$ Using the classical
density theorem in topological $K$-theory \cite[pg. 109]{Karoubi livre}, it is
easy to show that we get the same twisted $K$-theory as for $g-h$ compact. We
may also choose the transition functions to be C$^{\infty}$, as we did in the
finite dimensional case.

The computations in the finite dimensional case may now be easily transposed
in this framework if we consider transition functions $(g_{ji},h_{ji})$ in the
group\footnote{More precisely, in the group $Aut(P)\subset G=\mathcal{A}%
^{\ast};$ see below.} $G=\mathcal{A}^{\ast}$ and take \textquotedblleft
supertraces\textquotedblright\ instead of traces. We just have to be careful
that the fibres of our bundles are not necessarily free\footnote{As we
mentioned already in Section 4, the fibres should be free if $\lambda$ does
not define a torsion class in the cohomology of each connected component of
$X;$ see below.}. Concretely, we define a rank map%
\[
Rk=Ch_{(0)}:K_{\lambda}(X)\rightarrow H^{0}(X;\mathbb{Z})
\]
as follows: if $E$ is a finitely generated projective module over
\underline{$\mathcal{A}$}$_{\lambda},$ it is defined by a family of two
projection operators $(p_{0},p_{1})$ in the algebra \underline{$\mathcal{A}$%
}$_{\lambda}$. Then the trace of $p_{0}-p_{1}$ is a locally constant integer,
defining the rank function, since
\[
K(\mathcal{A})\cong K(\mathcal{K})\cong K(\mathbb{C})=\mathbb{Z}\mathrm{.}%
\]
If we look at $E$ as a twisted $\mathcal{A}$-bundle over $X$ with fibre $P$
(which is a finitely generated projective $\mathcal{A}$-module), we may
consider $\mathrm{End}(P)$ as included in $M_{n}(\mathcal{A})\cong\mathcal{A}$
and restrict the supertrace defined on $\mathcal{A}$ to $\mathrm{End}(P).$ For
instance, the suspertrace of the identity on $P$ is just the rank of $P.$ By
abuse of notations, we shall identify $\mathrm{End}(P)$ and its image in
$\mathcal{A}$.

\bigskip

We now define a connection on $E$ as a family of differential forms
$\Gamma_{i}=$ $(\Gamma_{i}^{0},\Gamma_{i}^{1})$ with values in $\mathrm{End}%
(P)\subset\mathcal{A\subset L\times L},$ such that $\Gamma_{i}^{0}-\Gamma
_{i}^{1}$ is a differential form with values in $L^{1},$ satisfying the same
compatibility condition as above
\[
\Gamma_{i}=g_{ji}^{-1}\cdot\Gamma_{j}\cdot g_{ji}+g_{ji}^{-1}\cdot
dg_{ji}+\omega_{ji}.1.
\]
We choose the transition functions $g_{ji}=(g_{ji}^{0},g_{ji}^{1})$ to be in
$\mathrm{Aut}(P)$ rather than\textrm{\ GL}$_{n}(\mathbb{C})$. Such connections
exist, for instance the barycentric connection considered in the finite
dimensional case%
\[
\Gamma_{i}=%
{\displaystyle\sum}
\alpha_{k}\cdot g_{ki}^{-1}\cdot dg_{ki},
\]
where $(\alpha_{k})$ is a partition of unity associated to the covering. The
only difference with the finite dimensional case is that $n$ is replaced by
$Rk(E)=Ch_{(0)}(E)$ and the usual trace by the supertrace\footnote{Note again
that the supertrace of "1" is the rank of $P$ which is positive or negative.}.
If we denote by $str$ this supertrace, we define:%

\[
Ch(E,\Gamma)=Ch_{(0)}(E)+%
{\displaystyle\sum\limits_{m=1}^{\dim(X)/2}}
\frac{1}{m!}str(%
{\displaystyle\sum\limits_{i}}
\alpha_{i}(R_{i})^{m}),
\]
where the $R_{i}$ are the local curvatures as functions of the $\Gamma_{i}$
defined above, and where $(\alpha_{i})$ is a partition of unity associated to
the given covering $\mathcal{U}$.

The computations made before in the finite dimensional case show as well that
$Ch(E,\Gamma)$ is a cocycle for the differential $D=d-.\theta.$ As in the
finite dimensional case, standard homotopy arguments also show that the
cohomology class of $Ch(E,\Gamma)$ is independent from the connection $\Gamma$
and from the partition of unity $(\alpha_{i})$ (see Appendix 8.3 for the details).

Therefore, for any $\lambda,$ the Chern character induces an isomorphism
between $K_{\lambda}(X)\otimes_{Z}\mathbb{R}$ and the twisted cohomology which
is the cohomology of the even part of the even/odd de Rham
complex\footnote{Note that $\Omega^{2k}(X)$ and $\Omega^{2k+1}(X)$ are the
real vector spaces of differential forms of degree $2k$ or $2k+1$ with values
in $(i)^{k}\mathbb{R}.$ The differential is the usual one $d$ on $\Omega
^{2k}(X)$ and $id$ on $\Omega^{2k+1}(X)$.} with the twisted differential
$D=d-.\theta.$ It is proved in \cite{Atiyah and Segal 2}, in a computation
involving again the exponential of even forms, that this twisted cohomology
depends only on the class of $\theta$ in the cohomology group $H^{3}%
(X;i\mathbb{R}).$

\bigskip Summarizing the previous discussion, we get the following theorem:

\begin{theorem}
Let $\mathcal{U}$ be a good covering of $X,$ $\lambda$ be a completely
normalized $2$-cocycle with values in $S^{1}$ associated to this covering. Let
$(\alpha_{i})$ be a partition of unity associated to this covering and let
$\theta$ be the $3$-differential form associated to $-\lambda,$ according to
Appendix 8.1. Then the Chern character%
\[
Ch:K_{\lambda}(X)\rightarrow H_{\theta}^{ev}(X;\mathbb{R})
\]
from twisted $K$-theory to even twisted cohomology induces an isomorphism%
\[
K_{\lambda}(X)\otimes_{Z}\mathbb{R}\cong H_{\theta}^{ev}(X;\mathbb{R}).
\]

\end{theorem}

\begin{remark}
The functoriality of the Chern character is discussed in Appendix 8.3. Its
multiplicative properties will be studied in the next theorem..
\end{remark}

\begin{remark}
One may also normalize the Chern character by putting a factor $(1/2\pi
i)^{r}$ in front of $Ch_{(r)}(E,\Gamma)$ and replace $\theta$ by $\theta/2\pi
i.$ Then we have to work with the usual de Rham complex, contrarily to our
convention in the Note $23.$
\end{remark}

If the space $X$ is formal in the sense of rational homotopy theory
\cite{Felix}, we may replace the de Rham complex by its cohomology viewed as a
graded vector space (with the differential reduced to $0$). In that case, the
(even) twisted cohomology is isomorphic to the even part of the cohomology of
the complex
\[
\left[  \oplus H^{2k}(X;(i)^{k}\mathbb{R})\right]  \oplus\left[  \oplus
H^{2k+1}(X;(i)^{k}\mathbb{R})\right]  ,
\]
with the differential given by the cup-product with the cohomology class of
$\theta$ in $H^{3}(X;i\mathbb{R}).$ By a well known and deep theorem of
Deligne, Griffith, Morgan and Sullivan \cite{Deligne}, this computation is
valid when $X$ is a simply connected compact K\"{a}hler manifold.

In the particular case when $\theta$ is not $0$ in all the cohomology groups
$H^{3}(X_{r};i\mathbb{R}),$ where the $X_{r}$ are the connected composents of
$X,$ we see by a direct computation that $Ch_{(0)}(E,\Gamma)$ is necessarily
$0,$ which implies that the fibres of $E$ should be free $\mathcal{A}%
$-modules. This also implies that $Ch_{(1)}(E,\Gamma)$ is a closed
differential form. Therefore, for any $\lambda,$ one can define the first
Chern character\footnote{which is also a Chern class.} $Ch_{(1)}(E,\Gamma)$ in
the (non twisted) cohomology group $H^{2}(X;i\mathbb{R}).$ However, we need
the twisted differential cycles for the total Chern character of $E.$

Let now $\mathcal{U}=(U_{i})$ and $\mathcal{V}$ $=(V_{j})$ be a covering of
$X$ and $Y$ respectively. Let $(\alpha_{i})$ (resp. $(\beta_{j}))$ be a
partition of unity associated to $\mathcal{U}$ (resp $\mathcal{V}$). The
products $(\alpha_{i}\cdot\beta_{j})$ define a partition of unity associated
to the covering $\mathcal{W}$ $=$ ($U_{i}\times V_{j})$ of $X\times Y.$

\begin{theorem}
Let $E$ be a $\lambda$-twisted $\mathcal{A}$-bundle on $X$ and let $F$ be a
$\mu$-twisted $\mathcal{A}$-bundle on $Y.$ Here $\lambda$ and $\mu$ are
explicit \v{C}ech cocycles $\lambda_{tsr}$ and $\mu_{wvu}$ with values in
$S^{1},$ associated to the coverings $\mathcal{U}$ and $\mathcal{V}$
respectively. Let $\overline{\lambda}$ and $\overline{\mu}$ be the closed
differential forms defined on each $U_{i}\times V_{j}$ by the formulas
\begin{align*}
\overline{\lambda} &  =%
{\displaystyle\sum\limits_{t,s}}
d\alpha_{t}\cdot d\alpha_{s}\cdot\lambda_{tsi}^{-1}\cdot d\lambda_{tsi}\\
\overline{\mu} &  =%
{\displaystyle\sum\limits_{w,v}}
d\beta_{w}\cdot d\beta_{v}\cdot\mu_{wvj}^{-1}\cdot d\mu_{wvj},
\end{align*}
as in Appendix 8.1. Then we have the commutative diagram\footnote{According to
the computations in Section 5, we identify the $K$-theory of $\mathcal{A}%
_{\lambda}\widehat{\otimes}\mathcal{A}_{\mu}$ with the $K$-theory of
$\mathcal{A}_{\lambda\mu}.$ However, in these computations, one has to replace
$\mathcal{K}$ by the ideal $L^{1}$ of trace class operators.}
\end{theorem}%

\[%
\begin{array}
[c]{ccc}%
K_{\lambda}(X)\times K_{\mu}(Y) & \rightarrow & K_{\lambda\mu}(X\times Y)\\
\downarrow &  & \downarrow\\
H_{\overline{\lambda}}^{ev}(X)\times H_{\overline{\mu}}^{ev}(Y) &
\longrightarrow & H_{\overline{\lambda}+\overline{\mu}}^{ev}(X\times Y)
\end{array}
\]

\begin{proof}
Let $\Gamma=(\Gamma_{i})$ (resp. $\nabla=(\nabla_{j}))$ be a connection on $E
$ (resp. $F).$ Then $\Delta=\Gamma\otimes1+1\otimes\nabla$ is a connection on
$E\otimes F.$ Therefore, if $R_{E}$ (resp. $R_{F})$ is the curvature
associated to $\Gamma$ (resp. $\nabla$),%
\[
R_{E\otimes F}=R_{E}\otimes1+1\otimes R_{F}%
\]
is the curvature associated to $\Delta$ over each open subset $U_{i}\times
V_{j}$ of $X\times Y.$Using the partition of unity $(\alpha_{i}\cdot\beta
_{j})$ associated to the covering $(U_{i}\times V_{j})$ and the binomial
identity, we find the relation
\[
\frac{1}{m!}Ch_{(m)}(E\otimes F,\Delta)=%
{\displaystyle\sum\limits_{p+q=m}}
\frac{1}{p!q!}Ch_{(p)}(E,\Gamma)Ch_{(q)}(F,\nabla),
\]
from which the theorem follows.
\end{proof}

Finally, we should add a few words concerning graded twisted $K$-theory which
is indexed essentially by elements%
\[
\left[  \widetilde{\lambda}\right]  \in H^{1}(X;\mathbb{Z}/2)\times
H^{2}(X;S^{1}).
\]
If we apply Theorem 4.4 of \cite{Karoubi}, this group (at least rationally) is
isomorphic to the ungraded twisted $K$-theory of $Y$, where $Y$ is the Thom
space of the orientation real line bundle $L.$ This $L$ corresponds to the
image of $\left[  \widetilde{\lambda}\right]  $ in $H^{1}(X;\mathbb{Z}/2)$. In
more precise terms, the graded twisted $K$-group tensored with the field of
real numbers is isomorphic to the odd twisted relative cohomology group of the
pair $(P,X).$ Here $P=$ $P(L\oplus1)$ denotes the real projective bundle of
$L\oplus1$ (with fibre $P^{1}\cong S^{1}),$ and the 3-dimensional cohomology
twist is induced by the projection $P\rightarrow X$ from the twist in the
cohomology of $X$. This (graded) twisted cohomology is different in general
from the twisted cohomology associated to the image of $\left[  \widetilde
{\lambda}\right]  $ in $H^{3}(X;i\mathbb{R}).$This is not surprising since the
usual real cohomology of a manifold with a coefficient system in
$H^{1}(X;\mathbb{Z}/2)$ also depends on this system.

\begin{remark}
If $A$ is not a commutative Banach algebra, there is no internal product%
\[
K_{n}(A)\times K_{p}(A)\rightarrow K_{n+p}(A)
\]
in general. Therefore, it is remarkable that such a product exists for twisted
$K$-groups which are $K_{\ast}(\underline{\mathcal{K}}_{\lambda}),$ where
$\underline{\mathcal{K}}_{\lambda}$ is a noncommutative Banach algebra..
\end{remark}

\section{Appendix}

\subsection{Relation between \v{C}ech cohomology with coefficients in $S^{1}$
and de Rham cohomology}

This section does not claim any originality. It may be easily deduced from the
classical books \cite{Bott-Tu}, \cite{Karoubi-Leruste} for instance, the basic
ideas going back to Andr\'{e} Weil. It is added for completeness' sake and a
normalization purpose.

Our first task is to make more explicit the cohomology isomorphism
\[
H^{r}(\mathcal{U})\mathbb{\cong}H_{dR}^{r}(X),
\]
where $\mathcal{U}$ is a good covering of $X$. The \v{C}ech and de Rham
cohomologies are here taken with coefficients in a real vector space of finite
dimension $V.$

Let us denote by $\Omega^{r}(X)$ the vector space of differential forms on $X$
with values in $V$ and let $(\alpha_{i})$ be a partition of unity associated
to the covering $\mathcal{U}$. We define a morphism\footnote{With $V$ provided
with the discrete topology.}%
\[
f_{r}:C^{r}(\mathcal{U};V)\rightarrow\Omega^{r}(X)
\]
in the following way. For $r=0,$ we send a cochain $(c_{i})$ to the
C$^{\infty}$ function%
\[
x\longmapsto%
{\displaystyle\sum}
\alpha_{i}(x)\cdot c_{i},
\]
which we simply write $%
{\displaystyle\sum\limits_{i}}
\alpha_{i}\cdot c_{i}.$ For general $r>0,$ we send the $r$-cochain
$(c_{i_{0}i1...ir})$ to the sum%
\[%
{\displaystyle\sum_{(i0,...,ir)}}
\alpha_{i0\cdot}d\alpha_{i1}\cdot...d\alpha_{ir}\cdot c_{i0i1...ir}.
\]
We have to check that this correspondence is compatible with the coboundaries,
i.e. that%
\[
f_{r+1}(\partial c)=d(f_{r}(c)).
\]
The cochain $\partial c,$ which we call $v,$ is defined by the usual formula%
\[
v_{i0i1...i(r-1)}=%
{\displaystyle\sum_{m=0}^{r+1}}
(-1)^{m}c_{i0...\widehat{im}...i(r-1)}.
\]
Therefore,%
\[
f_{r+1}(v)=%
{\displaystyle\sum_{(i0,...,i(r-1))}}
\alpha_{i0\cdot}d\alpha_{i1}\cdot...d\alpha_{i(r-1)}\cdot%
{\displaystyle\sum_{m=0}^{r+1}}
(-1)^{m}c_{i0...\widehat{im}...i(r-1)}.
\]
In the previous sum, the terms corresponding to an index $m>0$ are reduced to
$0$ since the sum of the corresponding $d\alpha$ is $0.$ The previous identity
may then be written%
\[
f_{r+1}(v)=%
{\displaystyle\sum_{(i0,...,i(r-1))}}
\alpha_{i0\cdot}d\alpha_{i1}\cdot...d\alpha_{i(r+1)}\cdot c_{i1...i(r+1)}%
\]%
\[
=%
{\displaystyle\sum_{(i1,...,i(r+1)}}
d\alpha_{i1}\cdot...d\alpha_{i(r+1)}\cdot c_{i1...i(r+1)},
\]
which is $d(f_{r}(c)),$ if we reindex the components of this sum: notice that
the $c^{\prime}s$ are constant functions.

The maps $(f_{r})$ define a morphism of complexes which is a quasi-isomorphism
over any intersection of the $U_{i}$ since the covering $\mathcal{U}$ is good.
Therefore, by a classical Mayer-Vietoris argument, they induce an isomorphism
between the \v{C}ech and de Rham cohomologies.

We take a step further and now compare the \v{C}ech cohomology $H^{r-1}%
(X:S^{1})$ with $H_{dR}^{r}(X)$ via a map%
\[
H^{r-1}(\mathcal{U};S^{1})\rightarrow H^{r}(\mathcal{U};V)\cong H_{dR}^{r}(X).
\]
This is the coboundary map associated to the exact sequence%
\[
0\rightarrow2i\pi\mathbb{Z\rightarrow}i\mathbb{R}\overset{e}{\rightarrow}%
S^{1}\rightarrow0,
\]
where $e$ is the exponential function and $V$ the real vector space
$i\mathbb{R}.$ If $\lambda_{i0i1..i(r+1)}\in Z^{r-1}(\mathcal{U};S^{1}),$
there is a cochain $u=$ $u_{i0i1..i(r+1)}$ such that $e(u)=\lambda$. The
classical definition of the coboundary map%
\[
H^{r-1}(\mathcal{U};S^{1})\rightarrow H^{r}(\mathcal{U};2i\pi\mathbb{Z})
\]
is as follows. We first consider the coboundary of $u$ in $C^{r}(\mathcal{U})
$, which we look as a cocycle with values in $2i\pi\mathbb{Z}$, defined by%
\[
c_{i0i1...ir}=%
{\displaystyle\sum_{m=0}^{r}}
(1)^{m}u_{i0...\widehat{im}...ir}.
\]
According to the previous considerations, the associated de Rham class with
values in $i\mathbb{R}=V$ is defined by%
\[
\omega=%
{\displaystyle\sum_{(i0,...,ir)}}
\alpha_{i0\cdot}d\alpha_{i1}\cdot...d\alpha_{ir}\cdot c_{i0...ir}%
\]%
\[
=%
{\displaystyle\sum_{(i0,...,ir)}}
\alpha_{i0\cdot}d\alpha_{i1}\cdot...d\alpha_{ir\cdot}%
{\displaystyle\sum_{m=0}^{r}}
(1)^{m}u_{i0...\widehat{im}...ir}.
\]
Using the same argument as above, this sum may be written
\[
\omega=%
{\displaystyle\sum_{(i1,...,ir)}}
d\alpha_{i1}\cdot...d\alpha_{ir}\cdot u_{i1...ir}.
\]
We notice that $\omega$ is a closed form since $c_{i0i1...ir}\in
2i\pi\mathbb{Z}.$ On the other hand, it is cohomologous up to the sign
$(-1)^{r}$ to the form
\[
\theta=%
{\displaystyle\sum_{(i1,...,ir)}}
\alpha_{i1}\cdot d\alpha_{i2}...d\alpha_{ir}\cdot du_{i1i2...ir}.
\]
Using again the fact that $c_{i0i1...ir}\in2i\pi\mathbb{Z},$ we see that
$\theta$ is equal on $U_{i0}$ to the following differential form%
\[%
{\displaystyle\sum_{(i1,...,ir)}}
\alpha_{i1}d\alpha_{i2}...d\alpha_{ir}\cdot du_{i1i2...ir}=%
{\displaystyle\sum_{(i2,...,ir)}}
d\alpha_{i2}...d\alpha_{ir}\cdot du_{i0i2...ir}.
\]
We observe that $du_{i0i2...ir}$ is the logarithmic differential of
$\lambda_{i0i2...ir}.$ Therefore, if we change the indices and $r$ to $r+1$,
we get the following theorem.

\begin{theorem}
Let $\lambda_{i0...ir}$ be an $r$-cocycle on a good covering $\mathcal{U}$
with values in $S^{1}$ and let $(\alpha_{i})$ be a partition of unity
associated to $\mathcal{U}$. Then the closed de Rham form $\omega$ of degree
$r+1$ with values in $V=i\mathbb{R}$ which is associated to $\lambda$ by the
coboundary map$\footnote{Notice that $\mathbb{R}$ is provided with the
discrete topology.}$%
\[
H^{r}(\mathcal{U};S^{1})\rightarrow H^{r+1}(\mathcal{U};2i\pi\mathbb{Z}%
)\rightarrow H^{r+1}(\mathcal{U};i\mathbb{R})
\]
is given by the following formula on each open set $U_i{}_0:$%
\[
\omega=(-1)^{r+1}%
{\displaystyle\sum\limits_{(i_{1},...,i_{r})}}
d\alpha_{i1...}d\alpha_{ir}\cdot(\lambda_{i0...ir})^{-1}\cdot d\lambda
_{i0...ir}.
\]

\end{theorem}

\bigskip\smallskip

\begin{ex}
If we choose $r=2$ as in our paper, those formulas may be simply written as%
\[
\omega=%
{\displaystyle\sum\limits_{(i,j,k)}}
d\alpha_{i}\cdot d\alpha_{j}\cdot d\alpha_{k}\cdot c_{ijk}%
\]
which is cohomologous to
\[
-\sum_{(i,j,k)}\alpha_{i}\cdot d\alpha_{j}\cdot d\alpha_{k}\cdot dc_{ijk}.
\]
On the other hand, for a fixed $l,$ if we consider the sequence $(l,i,j,k),$
and the fact that
\[
c_{ijk}-c_{ljk}+c_{lik}-c_{lij}\in2i\pi\mathbb{Z},
\]
we may replace $dc_{ijk}$ by $dc_{ljk}-dc_{lik}+dc_{lij}.$ Therefore, the
restriction of $\omega$ to $U_{l}$ may be written as%
\begin{align*}
\omega &  =-\sum_{(i,j,k)}\alpha_{i}\cdot d\alpha_{j}\cdot d\alpha_{k}\cdot
dc_{ljk}+\sum_{(i,j,k)}\alpha_{i}\cdot d\alpha_{j}\cdot d\alpha_{k}\cdot
dc_{lik}\\
&  -\sum_{(i,j,k)}\alpha_{i}\cdot d\alpha_{j}\cdot d\alpha_{k}\cdot dc_{lij}%
\end{align*}
or
\begin{align*}
-\sum_{(i,j,k)}\alpha_{i}\cdot d\alpha_{j}\cdot d\alpha_{k}\cdot dc_{ljk} &
=-\sum_{(j,k)}d\alpha_{j}\cdot d\alpha_{k}\cdot dc_{ljk}\\
&  =-\sum_{(j,k)}d\alpha_{j}\cdot d\alpha_{k}\cdot\lambda_{ijk}^{-1}%
d\lambda_{ljk},
\end{align*}
as a differential form on $U_{l}.$ If we assume the cocycle $\lambda$
completely normalized, we find the explicit formula given in Section 7.
\end{ex}

\subsection{ Some key isomorphisms between various definitions of twisted
$K$-groups}

We want to make more explicit the isomorphisms between the various definitions
of twisted $K$-theory given in Section 4. This is especially relevant to the
proof of the multiplicativity of the Chern character in Section 7.

With the notations of Section 4, the more basic one is probably the following%
\[
K(\underline{\mathcal{K}}_{\lambda})\overset{\cong}{\longrightarrow}%
K_{1}(\underline{\mathcal{B}}_{\lambda}/\underline{\mathcal{K}}_{\lambda}).
\]
We recall that the first group $K(\underline{\mathcal{K}}_{\lambda})$ is the
original definition of Rosenberg \cite{Rosenberg}. The second group may be
interpreted as the Fredholm definition of twisted $K$-theory as in
\cite{Atiyah and Segal} (or \cite{DK} if $\lambda$ defines a torsion class in
$H^{3}(X;\mathbb{Z})).$ More precisely, if $E$ is a $\lambda$-twisted Hilbert
bundle and if $\mathcal{F}(E)$ is the space of Fredholm maps in $\mathrm{END}%
(E),$ the map%
\[
\mathcal{F}(E)\rightarrow(\underline{\mathcal{B}}_{\lambda}/\underline
{\mathcal{K}}_{\lambda})^{\ast}%
\]
is a locally trivial fibration with contractible fibres, as we pointed out in
Section 4. Therefore, we have the identifications%
\[
K_{\lambda}(X)\cong K(\underline{\mathcal{K}}_{\lambda})\cong K_{1}%
(\underline{\mathcal{B}}_{\lambda}/\underline{\mathcal{K}}_{\lambda}).
\]

\begin{theorem}
Let $\tau$ be the generator of $K_{1}(\mathcal{L}/\mathcal{K})\cong%
\mathbb{Z},$ associated to the Fredholm operator given by the shift. Then the
cup-product with $\tau$ induces an isomorphism%
\[
\varphi:K(\underline{\mathcal{K}}_{\lambda})\overset{\cong}{\longrightarrow
}K_{1}(\underline{\mathcal{B}}_{\lambda}/\underline{\mathcal{K}}_{\lambda}).
\]

\end{theorem}

\begin{proof}
In this statement, we implicitly identify the Hilbert tensor product $H\otimes
H$ with $H.$ If we forget the twisting, there is a well defined ring map%
\[
\mathcal{K\otimes L}/\mathcal{K\rightarrow L}/\mathcal{K}.
\]
For the same reasons, there is a ring map%
\[
\underline{\mathcal{K}}_{\lambda}\widehat{\mathcal{\otimes}}\mathcal{L}%
/\mathcal{K\rightarrow}\underline{\mathcal{B}}_{\lambda}/\underline
{\mathcal{K}}_{\lambda}.
\]
When the base space $X$ varies, the cup-product with the element $\tau$
induces a morphism between the (twisted) $K$-theories associated to
$\underline{\mathcal{K}}_{\lambda}$ and $\underline{\mathcal{B}}_{\lambda
}/\underline{\mathcal{K}}_{\lambda}$ respectively (with a shift for the second
one). By a standard Mayer-Vietoris argument and Bott periodicity, we reduce
the theorem to the case when $X$ is contractible, which is obvious.
\end{proof}

\bigskip

Although we don't really need it in this paper, it might be interesting to
define explicitly the isomorphism backwards:%
\[
\psi:K_{1}(\underline{\mathcal{B}}_{\lambda}/\underline{\mathcal{K}}_{\lambda
})\overset{\cong}{\longrightarrow}K(\underline{\mathcal{K}}_{\lambda})\cong
K(\mathcal{A}_{\lambda}).
\]
Such a map $\psi$ is simply the connecting homomorphism in the Mayer-Vietoris
exact sequence in $K$-theory associated to the cartesian square%

\[%
\begin{array}
[c]{ccc}%
\underline{\mathcal{A}}_{\lambda} & \rightarrow & \underline{\mathcal{B}%
}_{\lambda}\\
\downarrow &  & \downarrow\\
\underline{\mathcal{B}}_{\lambda} & \rightarrow & \underline{\mathcal{B}%
}_{\lambda}/\underline{\mathcal{K}}_{\lambda}%
\end{array}
.
\]
In more detail: if $\alpha$ is an invertible element in the ring
$\underline{\mathcal{B}}_{\lambda}/\underline{\mathcal{K}}_{\lambda},$ we
consider the $2\times2$ matrix%
\[
\left(
\begin{array}
[c]{cc}%
\alpha & 0\\
0 & \alpha^{-1}%
\end{array}
\right)  .
\]
By the Whitehead lemma (or analytic considerations: see below), this matrix
may be lifted as an invertible $2\times2$ matrix with coefficients in
$\underline{\mathcal{B}}_{\lambda},$ say $\gamma.$ Let $\varepsilon$ be the
matrix defining the obvious grading%
\[
\varepsilon=\left(
\begin{array}
[c]{cc}%
1 & 0\\
0 & -1
\end{array}
\right)  .
\]
Then the couple $(\varepsilon,\gamma\cdot\varepsilon\cdot\gamma^{-1})$ defines
an involution $J$ on $M_{2}(\underline{\mathcal{A}}_{\lambda})\cong%
\underline{\mathcal{A}}_{\lambda},$ hence a finitely generated projective
module over $\underline{\mathcal{A}}_{\lambda}$ which is simply the image of
$(J+1)/2.$ It is easy to show that the class in $K(\underline{\mathcal{A}%
}_{\lambda})$ is independent from the choice of the lifting $\gamma:$ this is
the classical definition of the connecting homomorphism $\psi$ (see e.g.
\cite{Milnor}).

Instead of working with invertible elements $\alpha,$ we may as well consider
families of Fredholm maps $D$ mapping to $\alpha,$ which are already in
\underline{$\mathcal{L}$}$_{\lambda}.$ Without loss of generality, we may also
assume $\alpha$ unitary which implies that a lifting of $\alpha^{-1}$ may be
choosen to be the adjoint $D^{\ast}.$ We now write the identity%
\[
\left(
\begin{array}
[c]{cc}%
D & 0\\
0 & D^{\ast}%
\end{array}
\right)  =\left(
\begin{array}
[c]{cc}%
0 & D\\
-D^{\ast} & 0
\end{array}
\right)  \cdot\left(
\begin{array}
[c]{cc}%
0 & -1\\
1 & 0
\end{array}
\right)  .
\]
If we define $\nabla_{D}$ as
\[
\nabla_{D}=\left(
\begin{array}
[c]{cc}%
0 & D\\
-D^{\ast} & 0
\end{array}
\right)
\]
in general, we see that we may choose the element $\gamma$ above to be
$\exp(\pi\nabla_{D}/2)\cdot\nabla_{-1}.$ Therefore,%
\begin{align*}
\gamma\cdot\varepsilon\cdot\gamma^{-1} &  =\exp(\pi\nabla_{D}/2)\cdot
\nabla_{-1}.\varepsilon\cdot\nabla_{1}\cdot\exp(-\pi\nabla_{D}/2)\\
&  =-\exp(\pi\nabla_{D}/2).\varepsilon\cdot\exp(-\pi\nabla_{D}/2).
\end{align*}
On the other hand, it is clear that $\nabla_{D}$ and $\varepsilon$
anticommute. Therefore, the previous formula may be written as%
\[
\gamma\cdot\varepsilon\cdot\gamma^{-1}=\exp(\pi\nabla_{D}).\varepsilon.
\]
The couple
\[
J=(\varepsilon,\exp(\pi\nabla_{D}).\varepsilon)
\]
defines the required element of $K(\underline{\mathcal{A}}_{\lambda}).$ By
construction, we see that $J$ also defines an element of the relative group
associated to the augmentation map%
\[
(\underline{\mathcal{K}}_{\lambda})^{+}\rightarrow\mathbb{C}\mathrm{.}%
\]
Here $(\underline{\mathcal{K}}_{\lambda})^{+}$ is the ring $\underline
{\mathcal{K}}_{\lambda}$ with a unit added and the relative $K$-group is the
usual one:%
\[
K(\underline{\mathcal{K}}_{\lambda})=Ker(K((\underline{\mathcal{K}}_{\lambda
})^{+})\rightarrow K(\mathbb{C})=\mathbb{Z}\mathrm{)}%
\]
which is canonically isomorphic to $K(\mathcal{A}_{\lambda}).$

\subsection{Some functorial properties of twisted $K$-theory and of the Chern
character}

In this paper, we have indexed twisted $K$-theory by completely normalized $2
$-cocycles $\lambda$ with values in $S^{1}.$ Of course, such a cocycle
determines a cohomology class $\left[  \lambda\right]  $ in $H^{2}%
(X;S^{1})\cong H^{3}(X;2i\pi\mathbb{Z})$ as we have seen in $8.1$ and we would
like to index twisted $K$-theory by elements of this smaller group. There is
an obstruction to doing so however as we shall see. If we apply Proposition
1.2 to $\mathbb{C}$-bundles (if $\left[  \lambda\right]  $ is a torsion class)
or to $\mathcal{A}$-bundles in general, we see that if $\mu$ is cohomologous
to $\lambda,$ the equivalence $\Theta$ in this last proposition, between the
categories of $\lambda$-twisted bundles and $\mu$-twisted bundles, depends on
the choice of a cochain $\eta$ such that
\[
\mu_{kji}=\lambda_{kji}\cdot\eta_{ji}\cdot\eta_{ki}^{-1}\cdot\eta_{kj}.
\]
If $\eta^{\prime}$ is another choice, $\eta_{ji}.\eta_{ji}^{\prime-1}$ is a
one-dimensional cocycle with values in $S^{1}$. Since a one-dimensional
coboundary does not change $\lambda,$ we see that the ambiguity in the
definition of the previous category equivalence lies in the cohomology
group\footnote{We assume the covering good as in \ref{fine}.} $H^{1}%
(\mathcal{U};S^{1})\cong H^{2}(X;2i\pi\mathbf{Z})\cong H^{2}(X;\mathbf{Z}).$
In particular, the definition of twisted $K$-theory with coefficients in
$H^{3}(X;\mathbf{Z})$ has a well-defined meaning only if $H^{2}(X;\mathbf{Z}%
)=0.$

This remark is also important for the definition of the product%
\[
K_{\lambda}(X)\times K_{\mu}(X)\rightarrow K_{\lambda\mu}(X)
\]
which is detailed in many ways in Section 5. The Hilbert bundle $E_{\lambda},
$ defined at the beginning of this section, depends on the cocycle $\lambda.$
It depends on its cohomology class $\left[  \lambda\right]  $ up to a non
canonical isomorphism as we have just seen (except if $H^{2}(X;\mathbf{Z})=0)
$. Therefore, strictly speaking, we cannot define in a functorial way a
cup-product%
\[
K_{\left[  \lambda\right]  }(X)\times K_{\left[  \mu\right]  }(X)\rightarrow
K_{\left[  \lambda\mu\right]  }(X).
\]
Another remark is the choice of a good covering in order to define twisted $K
$-theory via twisted bundles. There is also a functorial problem since many
choices are possible. One way to deal with this is to show that the categories
of twisted bundles associated to different coverings give the same twisted
$K$-theory if we choose two \v{C}ech cocycles which are cohomologous. This is
again included in the contents of Proposition 1.2. As we already pointed out,
this identification is not canonical, except if $H^{2}(X;\mathbb{Z})=0.$

Let us now turn our attention to the definition of the Chern character. If we
fix the good covering $\mathcal{U}$, our definition depends heavily on the
choice of a partition of unity $(\alpha_{i}).$ If $(\beta_{i})$ is another
choice, there is a homotopy between them which is $t\longmapsto(1-t)\alpha
_{i}+t\beta_{i}.$ If $\lambda$ is a completely normalized $2$-cocycle with
values in $S^{1},$ the associated closed differential forms $\theta\alpha$ and
$\theta\beta$ are homotopic and therefore cohomologous: they define the same
class in $H^{3}(X;i\mathbb{R}).$ However, it is not completely obvious that
the associated twisted cohomologies $H_{\theta\alpha}^{ev}(X)$ and
$H_{\theta\beta}^{ev}(X)$ are isomorphic in a way compatible with the Chern
character. One way to deal with this problem is to consider $\lambda$-twisted
bundles over $X\times\left[  0,1\right]  $ with the partition of unity given
by $(1-t)\alpha_{i}+t\beta_{i}$ as above. We then have a commutative diagram
where the horizontal arrows are isomorphisms%
\[%
\begin{array}
[c]{ccccc}%
K_{\lambda}(X\times\left\{  0\right\}  ) & \longleftarrow & K_{\lambda
}(X\times\left[  0,1\right]  ) & \longrightarrow & K_{\lambda}(X\times\left\{
1\right\}  )\\
\downarrow &  & \downarrow &  & \downarrow\\
H_{\theta\alpha}^{ev}(X\times\left\{  0\right\}  ) & \longleftarrow &
H_{\theta}^{ev}(X\times\left[  0,1\right]  & \longrightarrow & H_{\theta\beta
}^{ev}(X\times\left\{  1\right\}  ).
\end{array}
\]
This diagram shows that the Chern character does not depend on the choice of
partition of unity up to canonical isomorphisms given by the horizontal arrows.

We cannot expect the Chern character to be functorial with respect to the
cohomology class of $\lambda$ in $H^{3}(X;\mathbb{Z}).$ However, it is
\textquotedblleft partially functorial\textquotedblright\ in the following
sense: if we choose a good refinement $\mathcal{V}=(V_{s})$ of $\mathcal{U}%
=(U_{i})$ as in Section 1, any restriction map of type
\[
\Theta_{\tau}:K_{\lambda}(\mathcal{U})\rightarrow K_{\mu}(\mathcal{V})
\]
(where $V_{s}\subset U_{\tau(s)})$ is an isomorphism. This isomorphism is not
unique and depends on $\tau,$ as it was pointed out in the proof of Propositon
1.3. If $(\beta_{s})$ if a partition of unity associated to the covering
$\mathcal{V}$ and $(\alpha_{i})$ a partition of unity associated to the
covering $\mathcal{U}$, the functions $(\alpha_{i}\cdot\beta_{s})$ define a
partition of unity associated to $\mathcal{U\cap V}$ which is just a
reindexing of the covering $\mathcal{V}$. On the other hand, we may also
reindex $\mathcal{U}$, in such a way that the functions $(\alpha_{i}\cdot
\beta_{s})$ define also a partition of unity of $\mathcal{U}$. Since the
twisted cohomology is homotopically invariant, it follows that the
\textquotedblleft restriction map\textquotedblright%
\[
H_{\overline{\lambda}}^{ev}(X)\rightarrow H_{\overline{\mu}}^{ev}(X)
\]
is also well defined and that the diagram%

\[%
\begin{array}
[c]{ccc}%
K_{\lambda}(\mathcal{U}) & \rightarrow & K_{\mu}(\mathcal{V})\\
\downarrow &  & \downarrow\\
H_{\overline{\lambda}}^{ev}(X) & \rightarrow & H_{\overline{\mu}}^{ev}(X)
\end{array}
\]
is commutative (with the notations of Theorem $7.2)$.

\end{document}